\documentclass[11pt]{amsart}
\usepackage{amssymb, amsthm, amsmath, amsfonts}
\usepackage[margin=1in]{geometry}
\usepackage{url}
\usepackage{enumerate}

\usepackage[usenames,dvipsnames]{xcolor}

\newcommand{\Z}{\mathbb Z}
\newcommand{\R}{\mathbb R}
\newcommand{\C}{\mathbb C}

\newcommand{\vep}[0]{\varepsilon}
\newcommand{\qr}[2]{\left(\frac{#1}{#2}\right)}

\newcommand{\im}{\operatorname{Im}}

\newcommand{\ld}{\lambda}

\newcommand{\smat}[4]{\left(\begin{smallmatrix} #1 & #2 \\#3 & #4\end{smallmatrix}\right)}
\newcommand{\pmat}[4]{\begin{pmatrix} #1 & #2 \\ #3 & #4 \end{pmatrix}}

\newtheorem{Thm}{Theorem}[section]
\newtheorem{Prop}[Thm]{Proposition}
\newtheorem{Lem}[Thm]{Lemma}
\newtheorem{Cor}[Thm]{Corollary}
\newtheorem*{Conj}{Conjecture}
\newtheorem*{unnamThm}{Theorem}

\newtheoremstyle{named}{}{}{\itshape}{}{\bfseries}{.}{.5em}{\thmnote{#3}}
\theoremstyle{named}

\theoremstyle{definition}
\newtheorem{Def}[Thm]{Definition}

\theoremstyle{remark}
\newtheorem{Rem}[Thm]{Remark}
\newtheorem{Ex}[Thm]{Example}

\renewcommand{\colon}{\, : \,}

\makeatletter
\def\imod#1{\allowbreak\mkern5mu({\operator@font mod}\,\,#1)}
\makeatother

\title{Arithmetic Properties of Picard-Fuchs Equations and Holonomic Recurrences}
\author{Zane Kun Li and Alexander W. Walker}

\address{Department of Mathematics, Princeton University, Princeton, New Jersey 08544}
\email{zkli@math.princeton.edu}

\address{Department of Mathematics, Brown University, Providence, Rhode Island 02912}
\email{alexander\_walker@brown.edu}

\thanks{The LSU Research Experience for Undergraduates Program is supported by a National Science Foundation grant, DMS-0648064.}

\begin{document}

\begin{abstract}
The coefficient series of the holomorphic Picard-Fuchs differential equation associated with the periods of elliptic curves often
have surprising number-theoretic properties. These have been widely studied in the case of the torsion-free, genus
zero congruence subgroups of index $6$ and $12$ (e.g. the \textit{Beauville families}).
Here, we consider arithmetic properties of the Picard-Fuchs solutions associated to general elliptic
families, with a particular focus on the index 24 congruence subgroups.  We prove that elliptic families with rational
parameters admit linear reparametrizations such that their associated Picard-Fuchs solutions lie in $\Z[[t]]$.  A
sufficient condition is given such that the same holds for holomorphic solutions at infinity. An
Atkin-Swinnerton-Dyer congruence is proven for the coefficient series attached to
$\Gamma_1(7)$.  We conclude with a consideration of asymptotics, wherein it is proved that many coefficient series
satisfy asymptotic expressions of the form $u_n \sim \ell \lambda^n/n$.
Certain arithmetic results extend to the study of general holonomic recurrences.
\end{abstract}
\maketitle

\section{Introduction and Summary of Results}
Let $f(t)$ be holomorphic on a region $\Omega$.  Then $f(t)$ is \emph{holonomic}
if it satisfies a differential equation $\mathcal{L}$ with coefficients in $\mathbb{C}[t]$.  The set
of holonomic functions is a commutative $\mathbb{C}$-algebra under addition and pointwise (Cauchy)
multiplication.  Similarly, a sequence $\{a_n\}$ is said to \emph{holonomic} if it satisfies a
linear recurrence with polynomial coefficients.  These two concepts are related: if
$f(t) = \sum a_n t^n$ is holonomic, then $\{a_n\}$ is holonomic.  The converse holds, but only
in the formal sense.  For example, $\{n!\}$ is clearly holonomic. For more general theory (especially
in the theory of computing with holonomic sequences and functions) see \cite{Mallinger}.

If the coefficients of $\mathcal{L}$ are polynomials in $\mathbb{Z}[t]$, then we say that $f$ is
\emph{holonomic over} $\mathbb{Z}$ (and similarly for $\mathbb{R}$, $\mathbb{C}$, etc.).  If $f(x)$ is holonomic
over $\mathbb{Z}[t]$ (or other rings), $f$ may or may not lie in $\mathbb{Z}[[t]]$ (consider the functions
$x$ and $e^x$).

The linear differential equations that will be of interest in this paper are
the so-called \emph{Picard-Fuchs} equations.  A proper introduction can be found in
\cite{BeukersStienstra}. Briefly, let $E_t$ be a smooth family of elliptic curves,
parametrized over $\mathbb{C}$. Fix a point $t_0$, and a cycle $\gamma$.  Locally, we may identify
$\gamma$ with cycles in an open neighborhood of fibers.  Now fix a holomorphic $1$-form
$\omega(t)$ defined over all smooth fibers, and define the \emph{period} of $\omega$ as
\[\pi(t) := \int_{\gamma} \omega(t).\]
It can be shown that $\pi(t)$ satisfies a second order differential equation with polynomial
coefficients, called the \emph{Picard-Fuchs equation} of $E_t$.  This differential equation
can be computed in full generality by the Griffiths-Dwork algorithm and by classical means in the case of elliptic curves.

Consider the family of elliptic curves characterized by the Weierstrass equation
\begin{align}\label{ellipfam}
E_t \colon y^2 + a_1 xy + a_3 y = x^3+ a_2 x^2 + a_4 x + a_6
\end{align}
in affine form, where $a_i=a_i(t) \in \mathbb{Z}[t]$.
We shall frequently write $a_i$ for both $a_i(t)$ and $a_i(0)$, the reduction of $a_i$ mod $t$.  To reduce ambiguity,
care will be made to differentiate between `$=$' and `$\equiv$'.  In terms of these coefficient functions, the
modular invariants $g_2$ and $g_3$ can be expressed as
\begin{align*}
12 g_2(t) &= \left( (a_1^2 + 4a_2)^2 - 24(a_1a_3 + 2a_4)\right), \\
-216 g_3(t) &= \left( (a_1^2+4a_2)^3 + 36(a_1^2+4a_2)(a_1a_3+2a_4) - 216(a_3^2 + 4 a_6) \right). \end{align*}
For convenience, we shall assume that the elliptic discriminant $\Delta(t) = g_2^3 - 27g_3^2$ is zero at the
origin (in any case, this can be accomplished by a simple reparametrization of $t$).
The $j$-invariant is written as $j(t) = g_2(t)^3/\Delta(t)$.

It is well known that in the case of elliptic curves, the Picard-Fuchs equation is given by the following system
of differential equations (see \cite[Equation (3.3)]{Stiller})
\begin{align}\label{pfdef}
\frac{d}{dt}\begin{pmatrix}f_{1}\\ f_{2}\end{pmatrix} = \frac{1}{24\Delta(t)}\begin{pmatrix}-2\Delta'(t) & 36\gamma(t)\\-3g_{2}(t)\gamma(t) & 2\Delta'(t)\end{pmatrix}\begin{pmatrix}f_{1} \\ f_{2}\end{pmatrix}
\end{align}
where $\gamma(t) = 3g_{3}(t)g_{2}'(t) - 2g_{2}(t)g_{3}'(t).$
From \cite{BeukersStienstra}, two independent solutions to the Picard-Fuchs equation at $t = 0$ are
$$\mathcal{F}(t) = \left( \frac{12g_2(t)}{12g_2(0)} \right)^{\!-1/4} \! \,\!_2F_1\left(\frac{5}{12},\frac{1}{12};1;\frac{1}{j(t)}\right)$$
and $\mathcal{F}(t)\log(1/j(t)) + \mathcal{G}(t)$, where $\mathcal{G}(t)$ is holomorphic about $t = 0$.\footnote{Moreover, these represent a basis for solutions in any neighborhood of $j(t)=\infty$.}
Let $f(t) = \sum_{n \geq 0} u_{n}t^{n}$ denote the unique holomorphic solution at $t = 0$ with $u_{0} = 1$.
Since $f(t)$ is a solution to a differential equation, the $\{u_{n}\}$ correspond to a holonomic recurrence.
We have the following theorem which bounds the denominator of $u_{n}$.

\begin{Thm}\label{genint}
Let $f = \sum_{n \geq 0} u_n t^n$ be the unique holomorphic solution to the Picard-Fuchs equation associated to the family
$E_t$ given in \eqref{ellipfam} and satisfying $u_0=1$.  With $d_n$ denoting the denominator of $u_n$ in
reduced form, $d_n$ divides the integer $\left(8 (12g_2(0))^4\right)^n$.
\end{Thm}

We remark that this theorem generalizes Theorem 1.5 of \cite{BeukersStienstra} in that Stienstra and Beukers specifically
consider the case when $a_{1} \equiv 1 \imod{t}$ and $a_{2} \equiv \cdots \equiv a_{6} \equiv 0 \imod{t}$.
A refinement of this theorem allows us to show that there exists a $k \in \mathbb{Z}$ such that $f(kt) \in \mathbb{Z}[[t]]$
and under certain additional conditions allows us to show that $f(t) \in \mathbb{Z}[[t]]$; that is, that the $\{u_{n}\}$
are integral, a result that would have been very hard to show from our recursive definition alone.  To further strengthen this,
we prove the following `reduction' theorem, which is ultimately dependent on the upper bound established in Theorem \ref{genint}.

\begin{Thm}\label{sharperub} Fix a prime $p$, and let $q_0, \ldots, q_\ell \in \Z[n]$.  Suppose that the integers $\{u_n\}$ satisfy the holonomic recurrence
\begin{align} \label{integerlevel1}
(n+1)^2 u_{n+1} = p^{k_0} q_0(n) u_n + p^{k_1} q_1(n) u_{n-1} + \ldots + p^{k_\ell} q_\ell(n) u_{n-\ell},
\end{align}
with $u_0=1$ (and we interpret $u_{-1}=\ldots = u_{-\ell}=0$) such that $k_i \geq (1+i)k_0$ for $1 \leq i \leq \ell$.
Then $p^{\lceil n(k_0 - 2/(p-1)) \rceil}$ divides $u_n$.
\end{Thm}

So far, our integrality results have relied upon explicit Weierstrass equations of the form (\ref{ellipfam}), in which the
coordinate functions $a_i(t)$ lie in $\Z[t]$.  While this is sufficient for many applications, it will often be necessary
to generalize this criterion to admit $a_i(t) \in \Z(t)$.  This is particular useful in conjunction with Top and Yui's
work \cite{TopYui}, in which explicit equations for higher index subgroups are obtained by rational maps of the parameter
$t$.  For example, the equations given for $\Gamma_0(8)$ read as $E_t\colon y^2=x^3+(2-t^2)x^2+x$.  Following the inclusion
$\Gamma_1(8) \subset \Gamma_0(8)$, the authors construct an explicit equation for $\Gamma_1(8)$ by
\[E'_t\colon y^2 = x^3 + (2- s^2)x^2+x, \quad s= \frac{t^2}{t^2+1}.\]
The associated Picard-Fuchs solution, when reparametrized $t \mapsto 4t$, appears to have a power series with integral coefficients,
where
\[\{u_{n}\}_{n = 0}^{\infty} = \{1, 0, 0, 0, 16, 0, -512, 0, 12864, 0, -299008, 0, 6743040 \ldots\}\]
yet our current extension of the Stienstra-Beukers method, Theorem \ref{genint}, does not permit rational parameters.

Let the rational function $r(t) = p(t)/q(t) \in \Z(t)$ be  given in lowest terms.  We define $\nu(r)=p(0)$ and $\delta(r)=q(0)$.
In general, if $q(0)\neq 0$, we have $r(t) \in \Z[[t/\delta(r)]]/\delta(r)$, where the outermost $\delta(r)$ can be dropped if
$r(0) \in \Z$.  With this new notation, we present the following generalization to Theorem \ref{genint}:

\begin{Thm}\label{strongint}
Suppose that the hypotheses of Theorem \ref{genint} are met, with the relaxed assumption that the coordinates
$a_i(t)$ lie in $\Z(t)$.  With $f(t) = \sum_{n \geq 0} u_n t^n$ and $d_n$ denoting the denominator of $u_n$, it follows
that $d_n$ divides the integer $(8\,\delta(\Delta) \nu(12g_2)^4)^n$.
\end{Thm}

We now consider Picard-Fuchs solutions at $t = \infty$. Using ideas from the proof of Theorem \ref{genint}, we give conditions under which the denominators of the coefficients of the Laurent series solution, in $t$, are bounded. Again under certain conditions, we can show that the coefficients of the Picard-Fuchs solution are integral.

As a motivating example for our analysis of holonomic recurrences and solutions to Picard-Fuchs equations,
we consider the solution and recurrences associated to $\Gamma_{1}(7)$. This subgroup is a genus zero,
torsion-free index 24 subgroup of $\mathrm{PSL}_{2}(\mathbb{Z})$, and is associated with the following family of curves:
\begin{align}\label{gamma17equation}
E_{t}\colon  y^{2} + (1 - t - t^{2})xy + (t^{2}+ t^{3})y = x^{3} + (t^{2} + t^{3})x^{2},
\end{align}
with Picard-Fuchs equation
\begin{align}\label{gamma17pf}
t(t+1)(t^3+8t^2+5t-1)F'' + (5t^4+36t^3+39t^2+8t-1)F' + (4t^3+21t^2+15t+1)F = 0.
\end{align}
Let $\sum_{n \geq 0}u_{n}t^{n}$
denote the unique holomorphic solution around $t = 0$. Then
\begin{equation}\label{gamma17recur}
\begin{aligned}
(n+1)^{2}u_{n+1} &=(2n+1)^{2}u_{n}+(13n^{2} + 2)u_{n-1} + (9n^{2} - 9n + 3)u_{n-2}+(n-1)^{2}u_{n-3}\\
u_{0} &= 1, u_{1} = 1, u_{2} = 6, u_{3} = 25, u_{4} = 125, u_{5} = 642.
\end{aligned}
\end{equation}
It follows easily from Theorem \ref{genint} that these $\{u_{n}\}$ are integral.  The family $E_t$ also admits a
holomorphic solution about $t= \infty$, which can be written as
$\sum_{n \geq 0}v_{n}t^{-n - 2}$ with
\begin{equation} \label{introtovn}
\begin{aligned}
(n+1)^{2}v_{n+1} &= -(9n^{2} + 9n + 3)v_{n} - (13n^{2} + 2)v_{n-1} - (2n-1)^{2}v_{n-2} + (n-1)^2 v_{n-3}\\
v_0 &= 1,v_{1} = -3,v_{2} = 12,v_{3} = -59,v_{4} = 325, v_{5} = -1908.
\end{aligned}
\end{equation}
Again, we can show that these $\{v_{n}\}$ are integral.  We can then prove that the sequence $\{|v_n|\}$
satisfies a certain congruence relation. Using the theory of modular forms and a theorem of Verrill in \cite{Verrill},
we obtain the following theorem.

\begin{Thm}\label{congthm2}
Let $\{v_{n}\}$ be defined such that
$$(n - 1)^{2}v_{n} = (9n^{2} - 27n + 21)v_{n - 1} - (13n^{2} - 52n + 54)v_{n - 2} + (2n - 5)^{2}v_{n - 3} + (n - 3)^{2}v_{n - 4}$$
with $v_{0} = 0$, $v_{1} = 1$, $v_{2} = 3$, $v_{3} = 12$, and $v_{4} = 59$.
Then for all primes $p \neq 7$ and any integers $m$ and $r$, we have
$$v_{mp^{r}} - \gamma_{p}v_{mp^{r - 1}} + \qr{p}{7}p^{2}v_{mp^{r - 2}} \equiv 0 \imod{p^{r}}$$
where $\gamma_{p}$ is the $p$th coefficient in the $q$-expansion
of $\eta(z)^{3}\eta(7z)^{3}$ and $\qr{p}{7}$ is a Legendre symbol.
\end{Thm}

To further study the arithmetic properties of holonomic recurrences in general, we
investigate asymptotics of these holonomic recurrences.
Consider the constant holonomic recurrence $\mathcal{U}$, given by
\begin{align} \label{asympeq1}
\mathcal{U}\colon u_{n+1} = \sum_{j=0}^N k_j u_{n-j}\;, \qquad k_j \in \C. \end{align}
We define $\chi(x)=x^{N+1} - k_0 x^{N} - \cdots - k_N$, typically called the \textit{characteristic polynomial}
of $\mathcal{U}$.  Let $\lambda_1,\ldots, \lambda_n$ denote the distinct roots of $\chi$.  It is well known that
there exist polynomials $p_{\lambda_i} \in \C[n]$ of degree strictly less than the multiplicity of $\lambda_i$
such that $u_n = \sum_{\lambda} \lambda_i^n p_{\lambda_i}(n)$.  Thus, the asymptotics of $\mathcal{U}$ are readily
obtained, and may be simplified even further by summing over only those eigenvalues which have maximum absolute value.
In the case of a single dominating eigenvalue $\lambda$ of multiplicity one, we obtain the particularly simple
asymptotic $u_n \sim p_{\lambda} \lambda^n$.

Far less is known about the asymptotics of general holonomic recursions (i.e., when $k_i \in \C(t)$ in \eqref{asympeq1}).
An important subclass of such recursions are those of \textit{Poincar\'e type}, in which $\lim_{n \to \infty} k_i(n)$ is
finite for all $i$; in that case we have the following theorems of Poincar\'e and Perron (\cite[Theorems 8.9-8.11]{Elaydi}):

\begin{unnamThm}
Suppose that $k_i(n) \neq 0$ for all $n \in \Z^+$.  Then $\mathcal{U}$ has a fundamental
set of solutions $\{u^1,\ldots, u^{N+1}\}$ such that
$\limsup_{n \to \infty} \sqrt[n]{| u_n^{i} |} = |\lambda_i |.$
If the eigenvalues $\lambda_i$ have distinct norms, then $\lim_{n \to \infty} u_{n+1}^i/u_n^i = \lambda_i$.
\end{unnamThm}

This theorem is of immediate interest due to the following proposition.

\begin{Prop} \label{picardispoincare}
Let $f(t) = \sum_{n \geq 0} u_n t^n$ represent the power series of the solution to a Picard-Fuchs differential system of
a family of elliptic curves.  Then the holonomic recurrence of the coefficients of $f$ is of Poincar\'e type.
\end{Prop}

Proposition \ref{picardispoincare} implies that the asymptotic growth of the coefficients $u_n$ in the
Picard-Fuchs solution $f(t) = \sum u_n t^n$ is approximately $\vert u_n \vert \sim \lambda^n$, in that for
all $\varepsilon > 0$ we have $(\vert \lambda \vert - \varepsilon)^n \ll \vert u_n \vert \ll (\vert \lambda \vert + \varepsilon)^n$
(where `$\ll$' denotes the asymptotic less than).  In one special case, we may say more:

\begin{Thm} \label{boundingaround}
Let $p_k(n) = a_k n^2 + b_k n + c_k \in \Z[n]$ with $a_k \geq 0$, and consider the sequence $\{u_n\}$, defined recursively by
\begin{align} \label{asymp0}
(n+1)^2 u_{n+1} = \sum_{k=0}^N p_k(n) u_{n-k}.
\end{align}
Then the characteristic polynomial $\chi$ has a unique positive real root $\lambda$ with multiplicity $1$.
If $a_k(k-1) +b_k \leq 0$ for all $k$, then there exists an $\ell > 0$ such that $u_n <  \ell \lambda^n/n$
for $n \gg 0$.
\end{Thm}

\begin{Ex} \label{gamma17asymp}
We illustrate the preceding with the family of elliptic curves associated to $\Gamma_1(7)$, first presented
in \eqref{gamma17equation}.  Given the recurrence
\begin{align*}
(n+1)^{2}u_{n+1} &=(2n+1)^{2}u_{n}+(13n^{2} + 2)u_{n-1} + (9n^{2} - 9n + 3)u_{n-2}+(n-1)^{2}u_{n-3} \\
u_{0} &= 1, \; u_{1} = 1,\; u_{2} = 6, \;u_{3} = 25,\; u_{4} = 125,\; u_{5} = 642, \end{align*}
we note that the condition $a_k (k-1)+ b_k \leq 0$ is met for all $k$ (in this case, equality holds for all $k$,
which will become significant in Theorem \ref{bounds2}).  Our unique positive root is $\lambda \approx 6.295897$,
a root of $x^3-5x^2-8x-1$.  With the notation of the proof of
Theorem \ref{boundingaround}, let $R=21$ and $\gamma = \lambda$, and choose $n=10^6$.  This permits $\beta = 507/19$,
whereby we may also take $\ell = 0.3556365$, so $u_n \leq 0.3556365 \lambda^n/n$ for $n \gg 0$.  We will typically be
interested in the infimum over all possible $\ell$ in Theorem \ref{boundingaround}, which can be approximated non-rigorously
through an increasing sequence on $n$. This example can be strengthened with the following theorem:
\end{Ex}

\begin{Thm} \label{bounds2}
Suppose that the conditions of Theorem \ref{boundingaround} hold, with $a_k(k-1) + b_k =0$ for all $k$ and
$c_k+ka_k \geq 0$ for all $k$.  Suppose further that $\lambda$ is unique in absolute value, and that $u_n \gg 0$
for large $n$.  Then $u_n \sim \ell_0 \lambda^n/n$, for some $\ell_0 \in \R^+$.
\end{Thm}

\begin{Ex}\label{bounds2ex}
Although the hypotheses of Theorem \ref{bounds2} may seem hopelessly restrictive, we note that these
conditions are met by the recursive sequence $\{u_n\}$ associated with $\Gamma_1(7)$ (see Example \ref{gamma17asymp}).
Indeed, with notation as in the proof of Theorem \ref{bounds2}, we have $r_0=1$, $r_1=15$, $r_2 =21$, and $r_3 = 4$.
It follows that $u_n \sim \ell_0 \lambda^n/n$, where $\lambda$ is as in Example \ref{gamma17asymp} and $0.35561 < \ell_0 < 0.35564$.
These bounds are obtained from Example \ref{gamma17asymp} and the equation $v_n > w_n$, coupled with the closed form
for $w_n$ (taking $M =1000$) where $v_{n}, w_{n}$, and $M$ are defined as in the proof of Theorem \ref{bounds2}.
We note that these bounds can be taken to arbitrary precision with our current methods.
\end{Ex}

Theorem \ref{bounds2} is useful for general holonomic recurrences, however, it does not allow us to explicit
determine the value of $\ell_{0}$. If we restrict to studying recurrences that
arise from Picard-Fuchs equations and hence from modular forms, one can actually deduce the asymptotic formula
for $u_{n}$ with the value of $\ell_{0}$ differently. To this end, we offer the following theorem.

\begin{Thm}\label{modform_bound}
Let $f(t) = \sum_{n \geq 0}u_{n}t^{n}$ denote the solutions to the Picard-Fuchs equation
and suppose the hypotheses of Theorem \ref{boundingaround} are satisfied. If the sequence $\{u_{n}\ld^{-n}\}$
is eventually positive and monotonically decreasing and for some constants $a$ and $b$ we have $f(t) + a\log(1 - \ld t) \rightarrow b$ as $t \rightarrow (\ld^{-1})^{-}$,
then $u_{n} \sim a\ld^{n}/n$.
\end{Thm}

Typically to find such constants $a$ and $b$, one resorts to the theory of modular forms.
As an application of this theorem, for the recurrence given in \eqref{gamma17recur}, we have
$u_{n} \sim a \lambda^{n}/n$ where $\lambda$ is as in Examples \ref{gamma17asymp} and \ref{bounds2ex} and
\begin{align}\label{g17_asymconst}
a = \frac{7\sin(4\pi/7)}{256\pi\sin^{3}(6\pi/7)\sin^{5}(2\pi/7)} \approx 0.3556270700876065.
\end{align}
Analogous constants in the asymptotics associated to other congruence subgroups are given, the method of proving these constants
is similar to that of the $\Gamma_{1}(7)$ case.
It seems likely that these constants are related to periods of certain elliptic curves.

Similar work has been done with the index 12 subgroups (groups associated to the
Beauville families), which offer the advantage of closed forms for the coefficients $\{u_n\}$.
Our analysis attempts to do as much as possible without the crutch of a closed form solution
to the recurrence.

Our paper is organized as follows. In Section \ref{initbound}, we prove Theorem \ref{genint}
and introduce the concept of ``integral level." Sufficient conditions for integrality
of certain holonomic recurrences are given. Theorem \ref{sharperub} is shown in Section \ref{integerlevel2}.
This is followed by Section \ref{congsub} where we consider the integral levels of $\Gamma_{1}(7)$ and $\Gamma_{0}(12)$.
In Section \ref{integralrational}, we extend Theorem \ref{genint} to the case when $a_{i}(t) \in \Z(t)$
and prove Theorem \ref{strongint}. We then consider the solution to the Picard-Fuchs equation at $t = \infty$
in Section \ref{pfinfinity}. The congruence in Theorem \ref{congthm2} associated
to the solution of the Picard-Fuchs solution at infinity is proven in Section \ref{congg17}. Finally, in Section \ref{asymptotics}
we study the asymptotics of holonomic recurrences, wherein Proposition \ref{picardispoincare}, Theorem \ref{boundingaround}, Theorem
\ref{bounds2}, and Theorem \ref{modform_bound} are proven.

The computations for finding Picard-Fuchs differential equations (based off \eqref{pfdef}) and the corresponding $\{u_{n}\}$
were done using a program the authors built in \textsc{Mathematica}. This program is available at \cite{pfprog}.

\section{Near-Integrality}\label{nearint}

Let $E_t$ be a family of elliptic curves such that $\Delta(0)=0$ (where $\Delta$ is the elliptic
discriminant).  From \cite{BeukersStienstra} (and earlier), if $E_t$ is parametrized with coefficients
in $\Z[t]$, then there exists a unique holomorphic solution $f(t)= \sum_{n \geq 0} u_n t^n$ to the Picard-Fuchs
equation about $t=0$ with $u_0=1$. Furthermore, $f(t)$ is a \emph{$G$-function}; that is, there exists
a constant $A$ such that the denominator $d_n$ of $u_n$ satisfies $\vert d_n \vert \leq A^n$
(see \cite{Chudnovsky}).  In what follows, we extend a result of Stienstra and Beukers \cite{BeukersStienstra}
to prove an effective bound on $A$.

\subsection{Preliminary Bounds}\label{initbound}
We first develop some preliminary bounds on $A$. To do this, we will need the following lemma.
\begin{Lem} \label{binomiallemma}
For integral $k \geq 0$, we have
\[12^{3k}\binom{-5/12}{k} \binom{-1/12}{k} \in \Z\;.\]
\end{Lem}
\begin{proof}
Up to units, we expand the product as $12^k/(k!)^2 \prod_{j=0}^{k-1}  (5+12j)(1+12j)\;.$
For prime $p$, let $n_p$ and $d_p$ denote the exponent of $p$ in the numerator and denominator,
respectively, in the above.  Then $n_2 = 2k$, while Legendre's formula implies that
\[d_2 = 2\left(\lfloor k/2 \rfloor + \lfloor k/4 \rfloor + \lfloor k/8 \rfloor + \ldots \right) \leq \sum_{i=0}^\infty k/2^i = 2k\;.\]
Similarly, one may show that $n_3 = k \geq d_3$.  Now let $p > 3$ be prime.  We have $p \mid (5+12j)$
if and only if $j \equiv -5 (12^{-1}) \imod{p}$, and a similar formula holds for $(1+12j)$.  Running
through $j$ up to $k-1$, we have $n_p \geq 2 ( \lfloor (k-1)/p \rfloor + \lfloor (k-1)/p^2 \rfloor + \cdots )$.
Furthermore, if $k \equiv 0\imod{p}$, then $n_p(k) = n_p(k+1)$ as functions of $k$, as neither $(5+12k)$ nor $(1+12k)$
will be divisible by $p$.  It follows that $n_p \geq 2 ( \lfloor k/p \rfloor + \lfloor k/p^2 \rfloor + \cdots )$.
As the right hand sum is $d_p$, we conclude that $n_p \geq d_p$ for all primes $p$, whence our product is integral.
\end{proof}

\begin{proof}[Proof of Theorem \ref{genint}]
Throughout, let $F(a,b,c;z)$ denote the hypergeometric function $\,\!_2F_1$.  From \cite{BeukersStienstra}, the function
\begin{align} \label{nearint2b}
f(t) &= \left( \frac{12g_2(t)}{12g_2(0)} \right)^{\!-1/4} \! F\left(\frac{5}{12},\frac{1}{12};1;\frac{1}{j(t)}\right) = \left(\frac{12g_2(t)}{12g_2(0)}\right)^{\!-1/4} \! \sum_{k=0}^\infty \binom{-5/12}{k} \binom{-1/12}{k} \left( \frac{1}{j(t)} \right)^k
\end{align}
gives a solution to the Picard-Fuchs differential equation for $E_t$ in a neighborhood of $j(t)=\infty$, which holds at
$t=0$ since $g_2(0) \neq 0$ and $\Delta(0)=0$.  By the preceding remarks, $12g_2(t) \in \Z[t]$, so we have
$12g_2(t)/(12g_2(0)) \in \Z[t/(12g_2(0))]$, using the fact that our constant term is $1$.  We recall that
$(1+8z)^{-1/4} \in \Z[[z]]$, hence $(12g_2(t)/12g_2(0))^{-1/4} \in \Z[t/(96g_2(0))]$ under
composition.\footnote{The function $(1-8z)^{-1/4}$ is the generating function for the sequence $\{a_n\}$ given by
$a_{n} := (2^{n}/n!)\prod_{k = 0}^{n - 1}(4k + 1)$.
Integrality of this sequence follows in the manner of Lemma \ref{binomiallemma}.}
(Note that all convergence properties are satisfied).

Next, note that $1/j(t) = 12^3 \Delta(t)/(12g_2(t))^3$ lies in $12^3 \Z[[t/(12 g_2(0))^4]]$.  To see this,
we observe that $(12g_2(t))^{-1} \in \Z[[t/(12g_2(0))]]/(12 g_2(0))$, hence $(12g_2(t))^{-3} \in \Z[[t/(12g_2(0))]]/(12g_2(0))^3$.
One may show that $\Delta(t) \in \Z[t]$, whereby $\Delta(t)(12g_2(t))^{-3} \in \Z[[t/(12g_2(0))]]/(12g_2(0))^3$ as well.
Since $\Delta(0)=0$, our constant term is $0$ and we have $1/j(t) = 12^3\Delta(t)/(12g_2(t))^3 \in 12^3 \Z[[t/(12g_2(0))^4]]$.
By Lemma \ref{binomiallemma}, it then follows by composition that $F(\frac{5}{12}, \frac{1}{12}, 1; 1/j(t)) \in \Z[[t/(12g_2(0))^4]]$.
Multiplication then gives $f(t) \in \Z[[t/(8\cdot (12g_2(0))^4)]]$, as claimed.
This completes the proof of Theorem \ref{genint}.
\end{proof}

\begin{Rem} \label{genrem1}
As $12g_2(t)$ lies in $\Z[t]$, Theorem \ref{genint} will never suffice to prove integrality; we
may prove, at most, that $f(8t) \in \Z[[t]]$, under the assumption $12g_2(0) = \pm1$.  In \cite{BeukersStienstra} on the other
hand, Stienstra and Beukers treat a less general case, in which $a_1 \equiv 1$ and $a_2 \equiv \ldots \equiv a_6 \equiv 0$.
With these assumptions integrality can be shown directly, upon careful examination of the expansion
\begin{align} \label{g2decomp}
12g_2(t)= \left(a_1^{-1} \cdot (1+4a_2a_1^{-2})^{\!-1/2} \cdot (1-24(a_1a_3+2a_4)(a_1^2+4a_2)^{-2})^{\!-1/4}\right)^{\!4}\,.
\end{align}
Herein, we have simply traded strength for generality.  In any case, we have the following:
\end{Rem}

\begin{Cor}\label{gencor}
Let $f = \sum_{n \geq 0} u_n t^n$ as before.  Then a finite number of primes divide elements of the denominator set $\{d_n\}$.
Moreover, there exists a $k \in \Z$ such that $f(kt) \in \Z[[t]]$.
\end{Cor}

\begin{Cor} \label{gencor2}
With the notation of Theorem \ref{genint}, suppose that $g_2(t)/g_2(0)$ is a fourth power in $\Z[[t]]$
and that $12g_2(0) = \pm 1$.  Then the Picard-Fuchs solution $f(t)$ lies in $\Z[[t]]$.
\end{Cor}

\begin{proof}
Our additional assumption implies that $\alpha(t)= (g_2(t)/g_2(0))^{-1/4} \in \Z[[t]]$, using the fact that
$\alpha(0)=1$.  For the hypergeometric term in equation \eqref{nearint2b}, note that $1/j(t) = 12^3 \Delta(t)/(12g_2(t))^3$
lies in $12^3\Z[[t]]$, whereby Lemma \ref{binomiallemma} implies $F(\frac{5}{12},\frac{1}{12}, 1; 1/j(t)) \in \Z[[t]]$ as well.
\end{proof}

\begin{Rem} \label{essenceofsb}
This last corollary is the essence of Stienstra and Beukers integrality result in Theorem 1.5 of \cite{BeukersStienstra},
wherein the decomposition of \eqref{g2decomp} is used to show that our `fourth power' hypothesis is met.
\end{Rem}

\begin{Ex} Consider the family of elliptic curves parametrized by
\[E_t\colon y^2 + 5 x y + (t^2+1)y = x^3 +(t-3)x^2 + x +t^3 + 1\;.\]
We compute $12g_2(t) = 1+104t-104t^2$ and $\Delta(0)=0$.  With $\alpha(t)=13t-13$, we may write $(12g_2(t))^{-1/4} = (1-8t\alpha(t))^{-1/4}$.
As in the footnote of the proof of Theorem \ref{genint}, this implies that $12g_2(t)$ has a fourth root in $\Z[[t]]$.  By Corollary \ref{gencor2},
it then follows that the Picard-Fuchs solution $f(t) = \sum_{n \geq 0} u_n t^n$ lies in $\Z[[t]]$.  The (integral) coefficient series $\{u_n\}$ satisfies
a $12$-term holonomic recurrence and begins with
\[1,\,-86,\,26856,\,-10713740,\,4757138560,\,-2243385809196,\,1099636896720096\ldots\]
This example is taken only to illustrate that the requirements $a_1 \equiv 1$ and $a_2 \equiv \ldots \equiv a_6 \equiv 0$ from
\cite{BeukersStienstra} are not strictly needed in elementary proofs of integrality.  In Example \ref{ex1}, a more general
example is introduced.
\end{Ex}

A parametrization of the family $E_t$ is said to be \emph{reduced} if $a_i(t) \in \Z[nt]$ for all $i$ implies that $n= \pm 1$.
Clearly, every family is equivalent to a reduced family with $\Delta(0)=0$, under linear transformation. By Corollary \ref{gencor},
every reduced family admits a reparametrization $t \mapsto kt$, $k \in \Z$, such that $f(kt) \in \Z[[t]]$.
\begin{Def}
If $E$ has a reduced
parametrization and $k$ is minimal (in absolute value) such that $f(kt) \in \Z[[t]]$, then we say that $\vert k \vert $ is the
\emph{integral level} of $E$.
\end{Def}
One may extend this with the additional notion of \emph{fractional integral level}, corresponding
to the cases in which $u_n/k^n \in \Z$ for all $n$.  In any case, these will be well-defined by our previous work, and should be
viewed relative to the reduced representative of an elliptic curve equation.  Similar adjectives will be used to describe the
associated recurrences in the coefficients $u_n$ of $f$.

\begin{Ex}\label{ex1} Consider the (reduced) family of curves parametrized by
\[E_t\colon y^2+2x y = x^3 + x^2 + t x + t\]
We calculate $\Delta(t)= -16t(4t^2-13t+32)$ and $g_2(t) = 16/3-4t$.  By Theorem \ref{genint}, we have $d_n \mid 2^{27n}$, or equivalently,
$f(2^{27}t) \in \Z[[t]]$.  The coefficient recurrence associated to $f(2^{27}t)$ is given by
\begin{align*}
(n+1)^2 u_{n+1} &= 2^{20}(36n^2+68n+9)u_n- 2^{44}(76n^2+104n-123) u_{n-1} - 2^{71} (16n^2-48n+35) u_{n-2},\\
u_0 &= 1,\; u_1=2^{20} 3^2,\; u_2 = 2^{38}105 \end{align*}
hence $\{u_n\}$ is an integer sequence, despite division by $(n+1)^2$ at every iteration.
To contrast this with the `average' behavior of similar holonomic sequences, see Zagier's work \cite{Zagier}.  We stress that this example
does not imply that the integral level of $E$ is $2^{27}$; however, given our previous work, it does show that the level of $E$ divides $2^{27}$.
On the other hand, one may check (by inspection) that the level of $E$ is at least $2^9$.
\end{Ex}

\begin{Rem}
The Weierstrass form of an elliptic curve is not unique up to isomorphism, but one may show that the $j$-invariant is, and that
the modular invariant $g_2$ is unique up to multiplication by a fourth power \cite{Silverman}.  Expanding this to families of
elliptic curves, the formula in equation \eqref{nearint2b} implies that the solution to the Picard-Fuchs equation is unique
up to linear transformation and  multiplication by a polynomial in $\Z[t]$.  To avoid this technical matter in what follows,
explicit equations will be given for our parametrized families of curves.  Note, in particular, that any scaling of the argument
of $f$ (e.g. during the integral level analysis), corresponds to a trivial action on the isomorphism classes of families of elliptic curves.
\end{Rem}


\subsection{Further Bounds}\label{integerlevel2}
In this section we give further bounds on the integral level of $E_{t}$. We first give the proof of Theorem \ref{sharperub}.
\begin{proof}[Proof of Theorem \ref{sharperub}]
For integer $n$, let $\rho(n)$ denote the maximal power of $p$ dividing $n$.  By \eqref{integerlevel1}, we have
\begin{align} \label{integerlevel2eq}
2\rho(n+1) +\rho(u_{n+1}) &\geq \min\big(\rho(u_n) + k_0, \rho(u_{n-1})+k_1,\ldots, \rho(u_{n-\ell}) + k_\ell\big) \nonumber\\
&\geq k_0 + \min\big(\rho(u_n), \rho(u_{n-1})+k_1-k_0,\ldots, \rho(u_{n-\ell}) + k_\ell-k_0\big)
\end{align}
As a special case, suppose that the minimum in \eqref{integerlevel2eq} is always obtained by the first term, $\rho(u_n)$.  Then
\begin{align*}
\rho(u_{n+1}) &\geq k_0-2 \rho(n+1) + \rho(u_n) \geq 2k_0 - 2 \rho(n+1) - \rho(n) + \rho(u_{n-1}) \geq \ldots \\
& \geq k_0(n+1) - 2 \sum_{i=0}^n \rho(n+1-i) + \rho(u_0) \geq k_0(n+1) -2 \rho\big((n+1)!\,\big)
\end{align*}
wherein we have used the `logarithmic' property of $\rho$ and the fact that $\rho(u_0)=0$.  By Legendre's formula, we have
$\rho((n+1)!) \leq (n+1)/(p-1)$, hence $\rho(u_{n+1}) \geq (n+1)(k_0 - 2/(p-1))$.  Integrality of $\rho$ then implies
that $\rho(u_{n+1}) \geq \lceil (n+1)(k_0 - 2/(p-1)) \rceil$.  This concludes the special case.  Next, suppose that
the minimum in the first iteration of \eqref{integerlevel2eq} is the term $\rho(u_{n-i})+k_i - k_0$, instead of $\rho(u_n)$.  Then
\begin{align*}
\rho(u_{n+1}) &\geq k_0 - 2 \rho(n+1) + (\rho(u_{n-i})+k_i-k_0) \geq (i+1)k_0 - 2\rho(n+1) + \rho(u_{n-i}) \\
&\geq (i+1)k_0 - 2\big( \rho(n+1) + \ldots + \rho(n+1-i)\big) + \rho(u_{n-i}) \end{align*}
Note that this is precisely the inequality we would have derived after $(i+1)$ iterations of the above induction.
In general, this deviation from the special case has the same effect (none) regardless of when it is implemented
in the algorithm.  As the index of $u_i$ falls below $\ell$, note that the recursive definition in \eqref{integerlevel1}
will begin to omit summands, beginning with $p^{k_\ell}q_\ell(n)u_{n-\ell}$, hence the algorithm will be forced to
conclude along the lines of the special case (instead of jumping over $u_0$ to negative indices).
This completes the proof of Theorem \ref{sharperub}.
\end{proof}

While the previous section gives a theoretical upper bound for the integral level of a family of elliptic curves,
Example \ref{ex1} suggests that our current bounds are far from optimal.

\begin{Cor} \label{refinedupperboundcor}
With $\{u_n\}$ as in \eqref{integerlevel1}, let $s=\lceil (k_0 -2/(p-1)) \rceil$.  Then $v_n :=  p^{-ns}u_n \in \Z$,
and the sequence $\{v_n\}$ satisfies the recurrence relation
\begin{align} (n+1)^2 v_{n+1} = p^{k_0-s} q_0(n) v_n + p^{k_1-2s} q_1(n) v_{n-1} + \ldots + p^{k_\ell-(\ell+1)s} q_\ell(n) v_{n-\ell}\;.
\end{align}
In particular, $f(p^{-s}t) \in \Z[[t]]$.
\end{Cor}

\begin{Ex}\label{ex1a} We continue Example \ref{ex1}.  Taking $s=18$, we find that $f(2^{27-18}t) = f(2^9t) \in \Z[[t]]$,
with coefficients $\{u_n\}$ satisfying the integral holonomic recurrence
\begin{align*}
(n+1)^2 u_{n+1} &= 2^{2}(36n^2+68n+9)u_n- 2^{8}(76n^2+104n-123) u_{n-1} - 2^{17} (16n^2-48n+35) u_{n-2}\,, \\
u_0 &= 1,\; u_1=36,\;  u_2=420. \end{align*}
It follows that $E$ has integral level exactly $2^9$, where the lower bound stems from Example \ref{ex1}.  While our
algorithm here gives definite results, note that this does reflect the general case.
\end{Ex}


\subsection{Some Applications}\label{congsub}
In this section, we apply our analysis to several curves associated to certain genus zero congruence subgroups of $\mathrm{SL}_2(\Z)$.
By Remarks \ref{genrem1} and \ref{essenceofsb} the following proposition is immediate:

\begin{Prop}[cf. (1.5) in \cite{BeukersStienstra}]\label{intp1}
Let $E_t$ be a family of elliptic curves, parametrized as in \eqref{ellipfam}.  Assume $\Delta(0) = 0$ and let $f(t) = \sum_{n \geq 0} u_{n}t^{n}$ be
the unique holomorphic solution to the Picard-Fuchs equation with $u_{0} = 1$. Suppose \eqref{ellipfam} reduces modulo $t$
to $y^2\pm xy \equiv x^3$. Then $u_{n} \in \mathbb{Z}$ for all $n \geq 0$.
\end{Prop}

We now apply the above proposition to prove the integrality of two sequences related to the congruence subgroups $\Gamma_1(7)$ and $\Gamma_0(12)$, respectively.
In the following, we shall always take $f(t) = \sum_{n \geq 0} u_{n}t^{n}$, with $u_0=1$.

\begin{Ex}\label{intex1}
Consider the family of elliptic curves associated to the congruence subgroup $\Gamma_{1}(7)$ given in \eqref{gamma17equation}
with the associated recurrence in \eqref{gamma17recur}.
Observe that \eqref{gamma17equation} satisfies all the conditions of Proposition
\ref{intp1}. Then $u_{n} \in \mathbb{Z}$ for all $n \geq 0$ and $E_t$ has integral level $1$.
\end{Ex}

\begin{Ex}\label{intex2}
Next, we consider the family of elliptic curves associated to $\Gamma_{0}(12)$. Again from \cite{TopYui},
the parametrized Weierstrass equations for this family of elliptic curves is given by
\begin{align}\label{intex2eq1}
E_t\colon y^{2}+(t^{2}+1)xy-t^{2}(t^{2}-1)y = x^{3}-t^{2}(t^{2}-1)x^{2}.
\end{align}
Accordingly, the Picard-Fuchs equation is
$$t(9t^{4}-10t^{2}+1)F'' + (45t^{4}-30t^{2}+1)F' + 12t(3t^{2}-1)F = 0.$$
The recurrence for this example is
\begin{equation}
\begin{aligned}\label{intex2rec}
(n+1)^{2}u_{n + 1} &= (10n^{2} + 2)u_{n - 1}-9(n - 1)^{2}u_{n - 3}\\
u_{0} &= 1, \; u_{1} = 0, \; u_{2} = 3, \; u_{3} = 0, \; u_{4} = 15.
\end{aligned}
\end{equation}
Once again, the hypotheses of Proposition \ref{intp1} are met, hence $u_n \in \Z$ for all $n \geq 0$.
\begin{Rem}
Another way to prove $u_{n} \in \Z$ is to observe that $u_{2n + 1} = 0$ for all $n$ and
$u_{2n} = \sum_{k = 0}^{n}\binom{n}{k}^{2}\binom{2k}{k}$; unfortunately, closed forms of this type (sums of products of binomial coefficients) have not been found for other recurrences derived from congruence subgroups of level greater than 12.\footnote{This technique appears to work in the case of $\Gamma_0(12)$ for reasons described in \cite{TopYui} -- that $\Gamma_0(12)$ may be obtained as a double cover of the congruence subgroup $\Gamma_0(6)$ (which is of level $12$).}
\end{Rem}
\end{Ex}


\subsection{Integral Analysis with Rational Functions}\label{integralrational}
In the previous sections, we assumed the $a_{i}(t) \in \Z[t]$. Now we generalize the work in Section \ref{initbound} to when $a_{i}(t) \in \Z(t)$.
\begin{proof}[Proof of Theorem \ref{strongint}]
We proceed along the lines of the proof of Theorem \ref{genint}. From the notational remarks
preceding the statement of Theorem \ref{strongint}, it follows that
$$12g_2(t)/(12g_2(0)) \in \Z[[t/\delta(12g_2/(12g_2(0)))]] \subset \Z[[t/(12g_2(0)\delta(12g_2))]]= \Z[[t/\nu(12g_2)]].$$
As before, we obtain $(12g_2(t)/(12g_2(0)))^{-1/4} \in \Z[[t/(8 \nu(12g_2))]]$.

Next, note that $(12g_2(t))^{-3} \in \Z[[t/\nu(12g_2)]]/\nu(12g_2(0))^3$, hence
\begin{align} \label{nearintlem1}
\frac{1}{j(t)} = 12^3 \Delta(t)(12g_2(t))^{-3} \in 12^3 \cdot \frac{\Z[[t/(\delta(\Delta) \nu(12g_2))]]}{\nu(12g_2)^3}\;.
\end{align}
As $\Delta(0)=0$, the external denominator in \eqref{nearintlem1} can be absorbed into the brackets, and we have (by Lemma
\ref{binomiallemma}) that the hypergeometric term of the Picard-Fuchs solution $f(t)$ lies in $\Z[[t/(\delta(\Delta) \nu(12g_2)^4)]]$.
Upon multiplication we derive $f(t) \in \Z[[t/k]]$, where $k = 8\delta(\Delta) \nu(12g_2)^4$.
This completes the proof of of Theorem \ref{strongint}.
\end{proof}

The bound given in the statement of Theorem \ref{strongint} may be at times quite large.  In any case, Theorem \ref{strongint} does
provide us with an existence result, along the lines of Corollary \ref{gencor}.  Moreover, we remind the reader that the results
of Section \ref{integerlevel2} apply equally well in our new situation, which will immediately reduce our bound on the integral
level to a manageable, if not optimal, magnitude.

Additionally, we note that Theorem \ref{strongint} implies Theorem \ref{genint}.  Indeed, when the coordinates $a_i(t)$ are
relegated to $\Z[t]$, we have $\nu(12g_2) = 12g_2(0)$ and $\delta(\Delta)=1$, and we may recover our earlier bound.  Our work
in Theorem \ref{strongint} also applies to Corollary \ref{gencor2}, which we now formalize as

\begin{Cor} \label{fourthpowercor} With the notation of Theorem \ref{strongint}, suppose that $g_2(t)/g_2(0) \in \Z(t)$ is a
fourth power in $\Z[[t]]$.  Then the Picard-Fuchs solution $f(t)$ lies in $\Z[[t/(\delta(\Delta) \nu(12g_2)^4)]]$.
\end{Cor}

Along the lines of Remark \ref{essenceofsb}, we now offer a sharper version of the Stienstra-Beukers bound given in
\cite{BeukersStienstra}.  While technically weaker than Corollary \ref{fourthpowercor}, its hypotheses are verifiable by inspection alone.

\begin{Cor} \label{ratcor2} Suppose that the family of elliptic curves $E_t$ meets the hypotheses of Theorem \ref{strongint}, with the additional assumptions that
\begin{enumerate}[(i)]
\item $E_t$ reduces mod $t$ to $y^2 + x y \equiv x^3$
\item $\delta(a_i)=\pm 1$ for $i=1,\ldots, 6$.
\end{enumerate}
Then the integral level of $E$ is $1$.
\end{Cor}

\begin{proof}
One may compute that $\delta(\Delta)$ divides $\delta(a_1)^6 \delta(a_2)^3 \delta(a_3)^4 \delta(a_4)^3 \delta(a_6)^2$ in
general, hence $(ii)$ implies that $\delta(\Delta)=1$.  By condition $(i)$, we find $\nu(12g_2)$ divides $a_1(0)^4 \delta(a_2)^2 \delta(a_3) \delta(a_4)$,
which again is $1$ by hypothesis.  Moreover, $(i)$ implies that Corollary \ref{fourthpowercor} applies, which now gives our desired result.
\end{proof}

\begin{Ex}\label{g18ex}
In the introduction, we mentioned a family of elliptic curves $E'_t$, associated to the congruence
subgroup $\Gamma_1(8)$. Explicitly, we have
\[E'_t\colon y^2 = x^3 + (2- s^2)x^2+x\;, \quad s= \frac{t^2}{t^2+1}\;.\]
We compute $\delta(\Delta)=1$ and $\nu(12g_2) = 16$, hence Theorem \ref{strongint} implies that $f(2^{19}t) \in \Z[[t]]$.  The coefficients $u_n$ of $f(2^{19}t)$ satisfy
\begin{align*}(n+1)^2 u_{n+1} =
- &2^{39}(2n^2-3n+1)u_{n-1} - 2^{74}(23n^2-118n+143)u_{n-3} \\
&- 2^{113}(7n^2-61n+130)u_{n-5}-2^{150} \cdot 3(n^2-12n+35)u_{n-7},
\end{align*}
and so an appeal to Corollary \ref{refinedupperboundcor} with $k_0=18$ (whence $s=16$) implies that $f(8t) \in \Z[[t]]$.  On the other hand,
$u_8 = 201/1024$, which implies that a scaling $t \mapsto 4t$ is necessary to ensure integrality.  Thus, the integral level of $E_t'$ is either $4$ or $8$.

A bit more can be said for this example.  Based on the recurrence relation on $\{u_n\}$, we deduce that $f$ is an even function.
Setting $v_n = u_{2n}$, we derive the following integral recurrence for $8^{2n}u_{2n} = 64^n v_n$:
\begin{align*}
(n+1)^2 v_{n+1} = 2^6(4n^2 + n) v_n - &2^{10}(23n^2-36n+12) v_{n-1} \\
&+ 2^{16}(14n^2-47n+38)v_{n-2}-2^{22}(3n^2-15n+18)v_{n-3}\;.
\end{align*}
This sequence admits a reduction (take $k_0=5$ in Corollary \ref{refinedupperboundcor}), whence we deduce that $8^n v_n \in \Z$. This implies that $4^n u_n \in \Z$ for all $n$, whereby $f(4t) \in \Z[[t]]$ and we have shown that the integral level of $E_t'$ is precisely $4$.  Indeed, we have the stronger result that $f(2\sqrt{2} t) \in \Z[[t]]$, which -- while not an integer scaling -- remains significant for number theoretic reasons.
\end{Ex}
%

\section{Picard-Fuchs Solutions at Infinity}\label{pfinfinity}

Let $E_t$ denote a family of parametrized elliptic curves, with $j$-invariant $j(t) = g_2(t)^3/\Delta(t)$.  Thus far we
have considered only the holomorphic solution to the Picard-Fuchs equations at $t=0$.  This naturally extends to power
series solutions centered at $t \in \Z$ (where $\Delta(t)=0$) by affine transformation, which will preserve integrality
in our power series coefficients (as a composition of functions in $\Z[[t]]$).  To extend these results to the case
$t=\infty$, we recall that the function
\begin{align} \label{curvefamily}
f(t) = (12g_2(t))^{-1/4} F\left(\frac{5}{12}, \frac{1}{12}, 1; \frac{1}{j(t)}\right)
\end{align}
of Theorem \ref{genint} gives a solution to the Picard-Fuchs equation for $E_t$ in any neighborhood of $j=\infty$ (although,
in general, this will not give $u_0=1$ as we have scaled differentially here compared to \eqref{nearint2b}). This, of course,
does not require $\Delta(t)=0$; instead, we take $t=\infty$ and
make throughout this section the assumption that $j(\infty)=\infty$.

\begin{Thm} \label{picardinf}
Consider the parametrized family of elliptic curves given by
\[E_t\colon y^2 + a_1(t) x y + a_3(t) y = x^3 + a_2(t) x^2 + a_4(t) x + a_6(t)\;,\]
where $a_i \in \Z[t]$. Let $\alpha_N$ denote the leading coefficient of $12g_2(t)$, and suppose that $j(\infty)=\infty$.
If $4 \mid \deg g_2$, then the Picard-Fuchs differential equation for $E_t$ admits a holomorphic solution of the form
$$f(t)=\sum_{n \geq 0} u_n t^{-n-(\deg g_2)/4}$$ in a neighborhood of $t=\infty$, in which we may take $u_{0} = 1$. Let
$d_n$ denote the denominator of $u_n$ in reduced form. Then $d_n$ divides $(8\alpha_N^4)^n$.
If $12g_2(1/t)t^{\deg g_2}/\alpha_N$
is a fourth power in $\Z[[t]]$, then we have the stronger result $d_n \mid \alpha_N^{4n}$.
\end{Thm}

\begin{proof}
Let $s=1/t$.  Then $(12g_2(1/s) s^{\deg g_2})^{-1} \in \Z[[s/\alpha_N]]/\alpha_N$, and hence
$(12g_2(1/s)/\alpha_N)^{-1} \in s^{\deg g_2} \Z[[s/\alpha_N]]$ by the remarks preceding the statement of Theorem \ref{strongint}.
Since $4 \mid \deg g_2$, it follows that $(12g_2(1/s)/\alpha_N)^{-1/4} \in s^{(\deg g_2)/4} \Z[[s/(8\alpha_N)]]$ as in
Theorems \ref{genint} and \ref{strongint}.  Moreover, if $(12g_2(1/s)s^{\deg g_2}/\alpha_N)$ is a fourth power in $\Z[[s]]$,
then the stronger result $(12g_2(1/s)/\alpha_N)^{-1/4} \in s^{(\deg g_2)/4} \Z[[s/\alpha_N]]$ also holds.

Next, we consider $1/j(1/s) = 12^3 \Delta(1/s)(12g_2(1/s))^{-3}$.  Within this factorization, we have $(12g_2(1/s))^{-3} \in s^{3\deg g_2} \Z[[s/\alpha_N]]/\alpha_N$
and $\Delta(1/s) \in s^{-\deg \Delta} \Z[[s]]$.  As $j(\infty)=\infty$, we have that $1/j(1/s) \to 0$ as $s \to 0$.  Thus, in the product
$\Delta(1/s)(12g_2(1/s)^{-3})$, we have $3 \deg g_2 - \deg \Delta>0$.  Using this, we conclude that
\[\frac{1}{j(1/s)} \in 12^3 s^{3 \deg g_2 - \deg \Delta} \Z[[s/\alpha_N^4]].\]
In other words, the external denominators arising from $(12g_2(1/s))^{-1}$ may be brought within the brackets.
Upon multiplication we derive $F(\frac{5}{12},\frac{1}{12}, 1; \frac{1}{j(1/s)}) \in \Z[[s/\alpha_N^4]]$, hence the
Picard-Fuchs solution $f(t)$ given by
\[f(t):= \left( \frac{12 g_2(t)}{\alpha_N} \right)^{-1/4} F\left(\frac{5}{12},\frac{1}{12}, 1; \frac{1}{j(t)}\right)\]
lies in $s^{(\deg g_2)/4} \Z[[s/(8 \alpha_N^4)]]$.  Furthermore, $f(t) \in s^{(\deg g_2)/4} \Z[[s/\alpha_N^4]]$ when $\alpha_N/(12g_2(t))$
is a fourth power in $\Z[[s]]$.  In either case, we may write $f(t) = \sum_{n \geq 0} u_n t^{-n - (\deg g_2)/4}$, in which
we now claim that $u_0=1$.  Our assumption $j(\infty)=\infty$ implies that $\lim_{s \rightarrow 0} F(\frac{5}{12},\frac{1}{12}, 1; \frac{1}{j(1/s)}) =1$,
hence $u_0 = \lim_{s \to 0} s^{(\deg g_2)/4} (12 g_2(1/s)/\alpha_N)^{-1/4}$.  As such, we may take $u_0 =1$ in the
expansion of $f(t)$ and need not worry about the effects of rescaling on integrality.
This completes the proof of Theorem \ref{picardinf}.
\end{proof}

Several of our previous results have relied on the hypotheses that $(g_2(t)/g_2(0))$ or $12g_2(t)/\alpha_N$ admit fourth roots in
$\Z[[t]]$ (see Corollaries \ref{gencor2} and \ref{fourthpowercor}; Theorem \ref{picardinf}).  Unfortunately, the authors are
unaware of any property or condition that is necessary and sufficient to show this result. As such we have only addressed two
special cases thus far:
\begin{enumerate}[(i)]
\item that which follows from $a_1 \equiv 1$ and $a_2\equiv \ldots \equiv a_6 \equiv 0$, or
\item that which follows from $(1+8t)^{-1/4} \in \Z[[t]]$.
\end{enumerate}
The following proposition, taken from a recent paper of Heninger, Rains, and Sloane \cite{HeningerRainsSloane}, can be used to extend these techniques:

\begin{Prop}[cf. Theorem 1 in \cite{HeningerRainsSloane}]
Fix an integer $n$, and define $\mu_n = n \prod_{p \mid n} p$, where the product ranges over the primes dividing $n$.
Let $c(t) = \sum c_k t^k \in \Z[[t]]$, with $c_0=1$.  Then $c$ and $c \imod{\mu_n}$ admit $n$th roots in $\Z[[t]]$
simultaneously.
\end{Prop}
From this we derive the following trivial corollary.
\begin{Cor} \label{HRScor}
Let $p(t) \in \Z[t]$ with $p(0)=1$, and suppose that the coefficients of $p$ reduce modulo $8$ to the fourth power
of a polynomial in $\Z[t]$.  Then $p(t)$ admits a fourth root in $\Z[[t]]$.
\end{Cor}

\begin{Ex} We recall from Example \ref{intex1} that the family of elliptic curves
\[E_t\colon y^2 + (1-t-t^2)x y + (t^2+t^3)y = x^3 + (t^2+t^3)x^2\]
corresponding to $\Gamma_{1}(7)$ has integral level $1$.  Based on the calculation
\[j(t) = -\frac{\left(t^8+12 t^7+42 t^6+56 t^5+35 t^4-14 t^2-4 t+1\right)^3}{12^3 t^7(t+1)^7 \left(t^3+8 t^2+5 t-1\right)}\;\]
and the fact that $\deg g_2 =8$, it follows that $E_t$ admits a holomorphic expansion about $t=\infty$, which
fits the form $f(t)= \sum_{n \geq 0} u_n t^{-n-2}$ by Theorem \ref{picardinf}.  In the notation of Theorem \ref{picardinf}
we have $\alpha_N =1$, hence the integral level of $E_t$ at infinity divides $8$.  To strengthen this, we appeal to
Corollary \ref{HRScor}; we have
\begin{align*}
12t^8 g_2(1/t)/\alpha_N &= t^8 - 4 t^7 - 14 t^6 + 35 t^4 + 56 t^3 + 42 t^2 + 12 t + 1\equiv (t^{2} + t + 1)^4 \imod{8},
\end{align*}
whence Corollary \ref{HRScor} implies that $12t^8g_2(1/t)/\alpha_N$ is a fourth power in $\Z[[t]]$.  By this, Theorem \ref{picardinf}
now shows that $f(t) \in \Z[[t]]$, so $E_t$ has integral level $1$ at infinity.  The associated recurrence of the
holonomic recurrence $\{u_n\}$ is given by
\begin{equation}
\begin{aligned}\label{infseq}
(n+1)^{2}u_{n+1} &= -(9n^{2} + 9n + 3)u_{n} - (13n^{2} + 2)u_{n-1} - (2n-1)^{2}u_{n-2} + (n-1)^2 u_{n-3} \\
u_0 &= 1,\, u_{1} = -3,\, u_{2} = 12,\, u_{3} = -59.
\end{aligned}
\end{equation}
This sequence (up to sign) will be studied in the next section.  One should note the large similarity between the
$\{u_n\}$ recursion and the recursion given in equation \eqref{gamma17recur}.  This is no accident; whenever the
Picard-Fuchs equation for a family $E_t$ admits singular expansions $f_0(t) = \sum u_n t^n$ and $f_\infty(t) = \sum v_n t^{-n -(\deg g_2)/4}$
about $t=0$ and $t=\infty$, respectively, the holonomic recursions $\{u_n\}$ and $\{v_n\}$ are related by the following proposition:
\end{Ex}

\begin{Prop}
Let $\{u_n\}$ and $\{v_n\}$ be as above.  If $\{u_n\}$ satisfies the finite holonomic recursive relation $\sum_k w_k(n) u_{n-k} = 0$
for some $w_k \in \Z[n]$, then $\{v_n\}$ satisfies the recursion $\sum_k w_k(-n)v_{n+k+\delta}=0$, where $\delta \geq 0$ is the minimal
integer such that $v_\delta \neq 0$. Under these recursions, the sequences $\{u_n\}$ and $\{v_n\}$ are uniquely defined by the conditions
$u_0=v_\delta=1$, and $u_n = v_m =0$ for $n<0$ and $m<\delta$.
\end{Prop}

\begin{proof} Let $p(t)f(t) + q(t)f'(t) + r(t)f''(t)=0$ denote the associated Picard-Fuchs equation, where $p(t)=\sum p_i t^i$, $q(t)= \sum q_i t^i$,
and $r(t)= \sum r_i t^i$ are polynomials in $\Z[t]$.  From the solution $f_0$ we derive
\[\sum_k \left( p_k - kq_{k+1} + r_{k+2}(k+1)k + n(q_{k+1}-r_{k+2}(2k+1))+n^2r_{k+2} \right) u_{n-k} =0\;,\]
seen by equating the $t^n$ coefficient to $0$ and regrouping.  Similarly, the solution $f_\infty$ implies
\[\sum_k \left( p_k - kq_{k+1} +r_{k+2}(k+1)k + n((2k+1)r_{k+2}-q_{k+1}) + n^2r_{k+2} \right) v_{n+\delta +k}=0\;,\]
upon consideration of the $t^{-n -\delta}$ term.  (For simplicity of notation, one should take these sums as infinite with
finite support.) Our first result follows by taking $w_k(n) = p_k - kq_{k+1} +r_{k+2}(k+1)k + n((2k+1)r_{k+2}-q_{k+1}) + n^2r_{k+2}$.
The initial conditions are trivial; merely stated for completeness.
\end{proof}

\subsection{A Congruence Associated to $\Gamma_{1}(7)$}\label{congg17}
We prove Theorem \ref{congthm2} which gives an Atkin-Swinnerton-Dyer congruence for the sequence (up to sign) given in \eqref{infseq}.
In this section, we let $q = e^{2\pi i \tau}$ where $\im(\tau) > 0$.
Note that
$$f(\tau) = q - q^{3} + 2q^{4} + 2q^{5} - 3q^{6} + q^{7} + 3q^{8} + \cdots = q\prod_{n = 1}^{\infty}(1 - q^{n})^{c_{n}}$$
where $c_{n} = 2, 0, 1, -2$ if $n \equiv 0, \pm 1, \pm 2, \pm 3 \imod{7}$, respectively, is a weight 1 modular form for $\Gamma_{1}(7)$.
The function $$t(\tau) = q - 3q^2 + 5q^3 - 6q^4 + 7q^5 - 7q^6 + 3q^7 + 4q^8 +\cdots = q\prod_{n = 1}^{\infty}(1 - q^{n})^{d_{n}}$$
where $d_{n} = 0, 3, -2, -1$ if $n \equiv 0, \pm 1, \pm 2, \pm 3 \imod{7}$, respectively,
is a Hauptmodul for $\Gamma_{1}(7)$ (cf. (4.23) in \cite{Elkies})\footnote{The Hauptmodul that Elkies takes is (in his notation) $y^{2}z/x^{3}$ where
$x$, $y$, and $z$ are defined as in (4.4) of his paper. Since the Hauptmodul is the generator of the function field of $X_{1}(7)$, we have
taken our Hauptmodul to be $-x^{3}/(y^{2}z)$.} and hence is weight 0 modular function. In the following we shall sometimes abuse notation
and write $f(q)$ and $t(q)$ instead of $f(\tau)$ and $t(\tau)$.

Define $\{v_{n}\}$ such that $f(\tau) = \sum_{n \geq 0}v_{n}t(\tau)^{n}$. Computation yields that the inverse series $q(t)$ of $t(q)$
is $q(t) = t + 3t^{2} + 13t^{3} + 66t^{4} + 365t^{5} + 2128t^{6} + 12859t^{7} + 79745t^{8} + \cdots$
and hence from how $f(q)$ is defined above, we have
$$f(t) = t + 3t^{2} + 12t^{3} + 59t^{4} + 325t^{5} + 1908t^{6} + 11655t^{7} + 73155t^{8} + \cdots.$$
The following proposition shows that these coefficients agree (up to sign) with the sequence from \eqref{infseq}.
\begin{Prop}\label{unprop}
The $\{v_{n}\}$ are such that
$$(n - 1)^{2}v_{n} = (9n^{2} - 27n + 21)v_{n - 1} - (13n^{2} - 52n + 54)v_{n - 2} + (2n - 5)^{2}v_{n - 3} + (n - 3)^{2}v_{n - 4}$$
with $v_{0} = 0$, $v_{1} = 1$, $v_{2} = 3$, $v_{3} = 12$, and $v_{4} = 59$.
\end{Prop}
\begin{proof}
Define the two differential operators $D_{t} = t\frac{d}{dt}$ and $D_{q} = q\frac{d}{dq}$.
Let $G_{1} = (D_{q}t)/t$, $G_{2} = (D_{q}f)/f$,
\begin{align*}
p_{1}(t) = \frac{D_{q}G_{1} - 2G_{1}G_{2}}{G_{1}^{2}},\quad\text{ and }\quad p_{2}(t) = -\frac{D_{q}G_{2} - G_{2}^{2}}{G_{1}^{2}}.
\end{align*}
Then by Theorem 1 of \cite{Yang04},
$D_{t}^{2}f + p_{1}D_{t}f + p_{2}f = t^{2}f'' + t(p_{1} + 1)f' + p_{2}f = 0.$
Since $t$ is a Hauptmodul, $p_{1}$ and $p_{2}$ are both rational functions of $t$.
Observe that
\begin{align*}
p_{1}(t) = \frac{2t^{4} + 4t^{3} - 9t + 2}{(t - 1)(t^{3} + 5t^{2} - 8t + 1)}\quad \text{ and }\quad p_{2}(t) = \frac{(t^{2} - t + 1)(t^{2} + 2t - 1)}{(t - 1)(t^{3} + 5t^{2} - 8t + 1)}.
\end{align*}
It follows that
\begin{align*}
t^{2}(t-1)(t^3+5t^2-8t+1)f''+t(3t^4+8t^3-13t^2+1)f'+(t^2-t+1)(t^2+2t-1)f=0
\end{align*}
and hence we have the desired recurrence. This completes the proof of Proposition \ref{unprop}.
\end{proof}
\begin{Rem}
One can also prove the above proposition by methods similar to the discussion on
Pages 58--60 of \cite{Beukers} and the values of $t(\tau)$ at the inequivalent cusps of $\Gamma_{1}(7)$
from \cite[Table 6]{KimKoo}.
\end{Rem}
%

To prove Theorem \ref{congthm2}, we use the following theorem of Verrill:
\begin{Thm}[cf. Theorem 1.1 in \cite{Verrill}]\label{verrillthm1}
Let $\Gamma$ be a level $N$ congruence subgroup for $\operatorname{SL}_{2}(\mathbb{Z})$.
Let $t(\tau)$ be a weight 0 modular function for $\Gamma$. Let $f(\tau)$
be a weight $k$ modular form for $\Gamma$, and let $g(\tau)$ be a weight
$k + 2$ modular form for $\Gamma$ which has an Euler product expansion.
Suppose that we can write $$f(\tau) = \sum_{n \geq 0} b_{n}t(\tau)^{n}\text{ and } g(\tau) = \sum_{n \geq 0} \gamma_{n}q^{n}$$ where $q = e^{2\pi i \tau}$.
Suppose that for some integers $M$ and $a_{d}$, $d \mid M$, we have
$$f(\tau)\frac{q}{t}\frac{dt}{dq} = \sum_{d \mid M} a_{d}g(d\tau),$$
then for a prime $p \nmid NM$ and integers $m$, $r$, we have
$$b_{mp^{r}} - \gamma_{p}b_{mp^{r - 1}} + \vep_{p}p^{k + 1}b_{mp^{r - 2}} \equiv 0 \imod{p^{r}}$$
where $\vep_{p}$ is the character of $g$. In particular if
$b_{1} = 1$, then $b_{p} \equiv \gamma_{p}\imod{p}$.
\end{Thm}

\begin{proof}[Proof of Theorem \ref{congthm2}]
In the notation of Theorem \ref{verrillthm1}, take $g(\tau) = \eta(\tau)^{3}\eta(7\tau)^{3}$ where $\eta(\tau)$ is the Dedekind
eta function. Since the weight 3 cusps forms for $\Gamma_{1}(7)$ (of which $g(\tau)$ is a member) has dimension 1, any modular
form in it is a Hecke eigenform and hence has an Euler product expansion. Also note that in this case, the character of
$g(\tau)$ is the real nontrivial character mod 7 and hence $\vep_{p} = \left(\frac{p}{7}\right)$. A finite computation
yields that
$$f(\tau) \cdot \frac{q}{t}\frac{dt}{dq} = g(\tau)$$
and hence we have
$$v_{mp^{r}} - \gamma_{p}v_{mp^{r - 1}} + \left(\frac{p}{7}\right)p^{2}v_{mp^{r - 2}} \equiv 0 \imod{p^{r}}$$
for all integers $m$ and $r$ and $p \neq 7$ with $\gamma_{p}$ the $p$th coefficient in the $q$-expansion of $\eta(\tau)^{3}\eta(7\tau)^{3}$.
This completes the proof of Theorem \ref{congthm2}.
\end{proof}

\section{Asymptotic Behavior} \label{asymptotics}
We now investigate asymptotics of $\{u_{n}\}$ associated to Picard-Fuchs equations and holonomic recurrences in general.
To do this, we first give the proof of Proposition \ref{picardispoincare}.

\begin{proof}[Proof of Proposition \ref{picardispoincare}]
Let $p = \sum p_i t^i$, $q = \sum q_i t^i$, and $r = \sum r_i t^i \in \C[t]$ denote the coefficients in the Picard-Fuchs
differential equation $p(t) f(t) + q(t) f'(t) + r(t) f''(t)=0$.  Equating coefficients for $t^n$ gives
\begin{align} \label{picardpoincareeq}
\sum_{i = 0}^{\deg(p)} p_i u_{n-i} + \sum_{i=0}^{\deg(q)} q_i (n+1 -i) u_{n+1 - i} + \sum_{i=0}^{\deg(r)} r_i (n+1-i)(n+2-i) u_{n+2-i} = 0
\end{align}
One may show that each zero of $\Delta(t)$ is a zero of $r(t)$, hence $r_0=0$ and the $u_{n+2}$ term in \eqref{picardpoincareeq} is thus not present in the
holonomic recursion for $\{u_n\}$.
We note that no $u_{n-i}$ coefficient in \eqref{picardpoincareeq} has degree (in $n$) exceeding $2$.
It now suffices to show $r_1 \neq 0$; it would then follow that the degree of the coefficient of $u_{n+1}$ (the leading term) will
not be strictly exceeded, hence $\{u_n\}$ will be of Poincar\'e type.

Suppose $r_1 = 0$, so $r$ has a root of order $2$ at $t=0$.  As the Picard-Fuchs equation is an ordinary differential equation of
Fuchsian type, we must have $q_0=0$ (else the differential equation would admit a
non-regular singular point at $t=0$). Taking $p$, $q$, and $r$ coprime without loss of generality, we then have $p_0 \neq 0$.
Equating the constant term in \eqref{picardpoincareeq} gives $p_0 u_0 + q_0 u_1 + r_0 u_2 =0$.  Thus $p_0 u_0 = 0$, and it follows
that $u_0=0$.  This contradicts the assumption that $u_0=1$ in the canonical holomorphic solution to the Picard-Fuchs equations.
Thus $r_1 \neq 0$, and we have shown that $\{u_n\}$ is of Poincar\'e type.
This completes the proof of Proposition \ref{picardispoincare}.
\end{proof}

\begin{proof}[Proof of Theorem \ref{boundingaround}]
We have $\chi(x) = x^{N+1} - a_0 x^{N} - \ldots - a_{N-1} x - a_N$.  As not all $a_i$ are $0$, we may factor $\chi$ as $\chi(x) = x^m p(x)$,
in which $p(0)$ is strictly negative.  Thus, for some $\varepsilon > 0$, we have $p(\varepsilon)<0$, hence $\chi(\varepsilon) < 0$.
As $\chi$ is asymptotically positive, let $\lambda$ denote the minimal positive root of $\chi$.  For $r=0,\ldots,N$ we have
\begin{align*}
\chi^{(r)}(\lambda) &= \frac{(N+1)!}{(N+1-r)!} \lambda^{N+1-r} - \sum_{k=0}^N a_k \frac{(N-k)!}{(N-k-r)!} \lambda^{k-r}\\
&= \frac{(N+1)!}{\lambda^r(N+1-r)!} \left(\lambda^{N+1} - \sum_{k=0}^N a_k \frac{(N-k)!}{(N+1)!} \cdot \frac{(N+1-r)!}{(N-k-r)!} \lambda^k \right) \\
& > \frac{(N+1)!}{\lambda^r(N+1-r)!}\cdot \chi(\lambda) = 0\;.
\end{align*}
As all derivatives of $\chi$ are non-negative on $[\lambda,\infty)$, it follows that $\chi$ is strictly increasing on that interval
and that $\lambda$ is the unique positive root of $\chi$ (with multiplicity $1$).  Next, define $v_n := n u_n/\lambda^n$, and let
$p_k(n) = a_k \cdot (n+1)(n-k) + r_k(n)$, so that $\deg(r_k) < 2$.  Then \eqref{asymp0} implies that the sequence $\{v_n\}$ satisfies
\[v_{n+1} = \sum_{k=0}^N \frac{p_k(n) v_{n-k}}{\lambda^{k+1} (n-k)(n+1)} = \sum_{k=0}^N \frac{a_k v_{n-k}}{\lambda^{k+1}} + \sum_{k=0}^N \frac{r_k(n) v_{n-k}}{\lambda^k (n-k)(n+1)}\;.\]
We note that the hypothesis $a_k(k-1) + b_k \leq 0$ for all $k$ implies that the polynomials $r_k(n)$ have a non-positive linear term.
Thus, take $R > 0$ such that $R \geq r_k(n)$ for all $n \geq 0$.  We have
\[v_{n+1} \leq \sum_{k=0}^N \frac{a_k}{\lambda^{k+1}} \max(v_n,\ldots, v_{n-N}) + \sum_{k=0}^N \frac{R \max(v_n,\ldots, v_{n-N})}{\lambda^{k+1} (n-k)(n+1)} \leq m_n + \frac{R(N+1)m_n}{\gamma (n-N)(n+1)}\;,\]
where $m_n = \max(v_n,\ldots, v_{n-N})$ and $\gamma = \min(\lambda,\ldots,\lambda^{N+1})$.  Note that the reduction of the first sum
to $m_n$ follows from the (algebraic) definition of $\lambda$.  Take $\beta > 0$ such that $v_{n+1} \leq m_n + \beta m_n/n^2$, and
choose $\ell > 0$ such that $v_n,\ldots, v_{n-N}$ satisfy $v_j < \ell e^{-2\beta/n}$ for some $n\geq \max(3,\beta/2)$ (allowed to
be arbitrarily large).  We claim that $v_{n+1} < \ell e^{-2\beta/(n+1)}$.  To see this, it suffices to show
$\ell e^{-2\beta/n} + \ell \beta e^{-2\beta/n}/n^2 < \ell e^{-2\beta/(n+1)}$, or equivalently,
\begin{align} \label{asymp1}
e^{-2\beta/(n+1)} - e^{-2\beta/n} > \beta e^{-2\beta/n}/n^2.
\end{align}
As the function $e^{-2\beta/x}$ is concave down on $x>\beta/2$, it follows that the left side of \eqref{asymp1} is greater than
$2\beta e^{-\beta/(n+1)}/(n+1)^2$, the derivative of $e^{-2\beta/t}$ at $t=n+1$, and it now suffices to show
$2 e^{2\beta/n} e^{2\beta/(n+1)} > (n+1)^2/n^2$.  For this, simply note that the left side is always greater than
$2$, while $(n+1)^2/n^2$ is less than $2$ for $n \geq 3$.  This gives the claim, hence $v_n < \ell e^{-2\beta/n}$
for large $n$.  In particular, we obtain $v_n \ll \ell$, hence $u_n < \ell \lambda^n/n$ for all $n$ large.
This completes the proof of Theorem \ref{boundingaround}.
\end{proof}

\begin{proof}[Proof of Theorem \ref{bounds2}]
We take all notation from Theorem \ref{boundingaround}.  Note that our additional assumptions imply that
$r_k(n) = r_k$ is a positive integer for all $k$.  We then have
\begin{align} \label{boundseq0}
v_{n+1} = \sum_{k=0}^N \frac{a_k v_{n-k}}{\lambda^{k+1}} + \sum_{k=0}^N \frac{r_k v_{n-k}}{\lambda^{k+1} (n-k)(n+1)}\;.
\end{align}
It follows that $v_{n+1} \geq \min(v_n,\ldots, v_{n-N})>0$ for large $n$, by simply ignoring the righthand sum in
\eqref{boundseq0}, henceforth denoted $R_{n+1}$.  With Theorem \ref{boundingaround} and the lower bound $v_{n+1} > \min(v_n,\ldots,v_{n-N})$,
this now implies that there exists an $N_1$ such that $0 < \alpha \leq v_n < \ell$ for $n > N_1$.  By this, we have
$\vert R_n \vert \leq A/n^2$ for some $A > 0$ and $n >N_1$.  Fix an $\varepsilon > 0$, and choose $N_2>N_1$ such that
$\vert R_n \vert < \varepsilon/n^{3/2}$ for $n>N_2$.  Now let $M>N_2+N$ ($M$ will be further specified later), and
define the sequence $\{w_n\}_{n \geq M-N}$ as follows:
\[w_{n+1} = \sum_{k=0}^N \frac{a_k v_{n-k}}{\lambda^{k+1}}\,,\qquad w_j=v_j \;\text{for } j=M,\ldots,M-N\;.\]
In other words, $\{w_n\}$ represents a linear approximation to $\{v_n\}$ (but with initial conditions that reflect the
non-linear nature of the recursion for $\{v_n\}$ for small $n$).  Next, we define $\delta_n := \vert w_n - v_n \vert$,
which satisfies the recursive inequality:
\begin{align} \label{boundseq1}
\delta_{n+1} = \vert w_{n+1} - v_{n+1} \vert &\leq \sum_{k=0}^N \frac{a_k}{\lambda^{k+1}} \delta_{n-k} + \vert R_{n+1} \vert \leq \max(\delta_{n},\ldots \delta_{n-N}) + \frac{\varepsilon}{(n+1)^{3/2}} \\ \label{boundseq2} \vspace{-1 mm}
& \leq \max(\delta_{M},\ldots ,\delta_{M-N}) + \sum_{k=-N}^{n-M} \frac{\varepsilon}{(M+k)^{3/2}} \leq \zeta(3/2)\varepsilon \;,
\end{align}
by the fact that $\delta_{M-j} =0$ for $j=0,\ldots,N$.  The manipulation of the maximum between \eqref{boundseq1} and \eqref{boundseq2}
is simply given by induction, using \eqref{boundseq1}.  As the sequence $\{w_n\}$ satisfies a truly linear recurrence, we may express
$w_n = \sum \rho_i(n) \omega_i^n$, wherein $\rho_i$ is a polynomial in $n$ and $\omega_i = \lambda_i/\lambda$, corresponding to a
scaling of the eigenvalues of $\chi$.  We have
\[\vert w_x - w_y \vert \leq \sum_{\omega_i} (\vert \rho_i(x) \vert + \vert \rho_i (y) \vert ) \cdot \vert \omega_i^x - \omega_i^y \vert \leq \sum_{\omega_i \neq 1} (\vert \rho_i(x) \vert + \vert \rho_i (y) \vert ) \cdot (\vert \omega_i^x \vert + \vert \omega_i^y \vert )\;.\]
By assumption, all $\omega_i \neq 1$ have magnitude not equal to $1$, and an application of Rouch\' e's Theorem then implies that
$\vert \omega_i \vert <1$ for $\omega_i \neq 1$. Thus there exists an $N_3>0$ such that $x,y > N_3$ implies $\vert w_x - w_y \vert \leq \varepsilon$.
Take $M > N_3$.  It then follows that
\[\vert v_n -v_m \vert \leq \vert w_n - v_n \vert + \vert v_n - v_m \vert + \vert w_m - v_m \vert \leq (2\zeta(3/2) + 1)\varepsilon\;,\]
where we choose $n,m > M$.  Thus $\{v_n\}$ represents a Cauchy sequence, and converges.  Moreover, Theorem \ref{boundingaround} now gives
$\lim_{n \to \infty} v_n = \ell_0 = \inf \ell$, where the infimum is taken over all possible choices of $\ell$ in Theorem \ref{boundingaround}.
Returning to the definition of $u_n$, we now have $u_n \sim \ell_0 \lambda^n/n$, as desired.
This completes the proof of Theorem \ref{bounds2}.
\end{proof}

Suppose that the recursion in Theorems \ref{boundingaround} and \ref{bounds2} arises as the coefficient recursion to a
Picard-Fuchs differential system.  Then the condition $a_k(k-1)+ b_k=0$ is equivalent to the assumption that the
Picard-Fuchs equation is of the form
\begin{align} \label{boundsrem}
A(t) F(t) + B'(t)F'(t)+B(t)F''(t)= A(t)F(t) + (B(t)F'(t))'=0\;,
\end{align}
with $A(t),B(t) \in \Z[t]$ (with $B(0)=0$).
We note that the elliptic families associated to the congruence subgroups $\Gamma_1(7)$, $\Gamma(8;4,1,2)$, and $\Gamma_1(10)$
correspond to differential systems of the form \eqref{boundsrem}.  For $\Gamma_1(7)$ and $\Gamma_1(10)$, Theorem \ref{bounds2}
applies directly.  For $\Gamma(8;4,1,2)$, a similar result will hold under the condition that the associated sequence $\{v_n\}$
can be bounded away from $0$ (since Theorem \ref{boundingaround} does not apply).

\begin{proof}[Proof of Theorem \ref{modform_bound}]
Let $v_{0} = 0$ and $v_{n} = u_{n} - a\ld^{n}/n$. Since $f(t) + a\log(1 - \ld t) \rightarrow b$
as $t\rightarrow (\ld^{-1})^{-}$, we obtain $\sum v_{n}t^{n} \rightarrow b$ for $t \rightarrow (\ld^{-1})^{-}$.
Our result in Theorem \ref{boundingaround} implies that $v_{n} = O(\ld^{n}/n)$. Thus, Littlewood's extension of Tauber's Theorem \cite[p.233--235]{titchmarsh}
(and a change of variables) gives $\sum v_{n}\ld^{-n} = b$, whence
\begin{align}\label{growthofsums}
\sum_{n \leq x}u_{n}\ld^{-n} = a\log x + b + \gamma + o(1).
\end{align}
Let $c(x) = u_{\lfloor x \rfloor}\ld^{-\lfloor x \rfloor}$ and define $C(x) = \int_{0}^{x}c(t)\, dt$. Since $\{u_{n}\ld^{-n}\}$ is eventually
positive and monotonically decreasing, we have
$$(s - r)xc(sx) \leq C(sx) - C(rx) \leq (s - r)xc(rx)$$
for sufficiently large $x$ and arbitrary real numbers $r$ and $s$ with $0 < r < s$. Since \eqref{growthofsums} implies that
$C(sx) - C(rx) = a(\log s - \log r) + o(1)$ as $x \rightarrow \infty$, it follows that
$$\limsup_{x \rightarrow \infty} xc(sx) \leq a\cdot \frac{\log s - \log r}{s - r}$$
for any positive $r < s$. For $s = 1$ and $r \rightarrow 1^{-}$, we obtain $\limsup_{x \rightarrow\infty} xc(x) = a$. Likewise, $$\liminf_{x \rightarrow \infty} xc(rx) \geq a\cdot \frac{\log s - \log r}{s - r}$$
for any positive $r < s$. With $r = 1$ and $s \rightarrow 1^{+}$, we obtain $\liminf_{x \rightarrow \infty} xc(x) = a$.
Thus $\lim_{x \rightarrow \infty}xc(x)$ exists and is equal to $a$, hence $u_{n} \sim a\ld^{n}/n$.
This completes the proof of Theorem \ref{modform_bound}.
\end{proof}

Theorem \ref{bounds2} do not allow us to determine $\ell_{0}$ exactly. If we are just working with a general holonomic recurrence,
then the only option remaining might be to compute $\ell_{0}$ to a very high accuracy. However, if we are working with a holonomic
recurrence which arises from a Picard-Fuchs differential equation, we can use Theorem \ref{modform_bound} and explicitly find $\ell_{0}$
through modular forms. The referee kindly offered us the proof of Example \ref{asym_g17} which applies Theorem \ref{modform_bound}
to show the asymptotic for the $u_{n}$ associated to $\Gamma_{1}(7)$. The method to prove Examples
\ref{asym_g8412} and \ref{asym_g110} should be similar.

\begin{Ex}\label{asym_g17}
Consider the elliptic family associated to $\Gamma_{1}(7)$. The equations parametrizing this family, the associated Picard-Fuchs equation, and
associated recurrence were stated in \eqref{gamma17equation}--\eqref{gamma17recur}.
With $f(t) = \sum_{n \geq 0}u_{n}t^{n}$ denoting the holomorphic solution of the Picard-Fuchs equation at $t = 0$
and $\ld \approx 6.295897$ (the dominant root of $x^{3} - 5x^{2} - 8x - 1$), we show that
$f(t) + a\log(1 - \ld t) \rightarrow b$ as $t \rightarrow (\ld^{-1})^{-}$ for some constants $a$ and $b$ which we give explicitly.

We first identify the modular function $t(\tau)$ and the modular form $f(\tau)$ associated to the differential equation \eqref{gamma17pf}. This can be done using
the idea described on Page 1298 of \cite{clausen_res}. We find
\begin{align*}
t(\tau) = \frac{E_{1}(\tau)^{2}E_{2}(\tau)}{E_{3}(\tau)^{3}}, \quad f(\tau) = \frac{E_{3}(\tau)\eta(\tau)^{2}}{E_{1}(\tau)^{3}E_{2}(\tau)^{5}}
\end{align*}
in which
$$E_{g}(\tau) = q^{7B(g/7)/2}\prod_{n = 1}^{\infty}(1 - q^{7(n - 1) + g})(1 - q^{7n - g})$$
denotes the generalized Dedekind eta function and $B(x) = x^{2} - x + 1/6$.

To continue, we determine the analytic behaviors of $t$ and $f$ at various cusps of $\Gamma_{1}(7)$, taken here as $\infty$, $2/7$, $3/7$, $0$, $1/2$, and $1/3$ (we need not consider the cusp class of $1/7$). For this, we consider the actions of the matrices $\smat{1}{0}{0}{1}$, $\smat{2}{-1}{7}{-3}$, $\smat{3}{-1}{7}{-2}$, $\smat{0}{-1}{7}{0}$, $\smat{7}{3}{14}{7}$, and $\smat{7}{2}{21}{7}$ on $t$ and $f$. By the transformation formulas for $E_{g}(\tau)$ given in \cite[Corollary 2]{yang_transf}, we have
\begin{align}
t(\tau)\bigg|\pmat{2}{-1}{7}{-3} &= -\frac{E_{2}(\tau)^{2}E_{4}(\tau)}{E_{6}(\tau)^{3}} = -q^{-1} + \cdots \label{comp1}\\
t(\tau)\bigg|\pmat{3}{-1}{7}{-2} &= \frac{E_{3}(\tau)^{2}E_{6}(\tau)}{E_{9}(\tau)^{3}} = -\frac{E_{3}(\tau)^{2}E_{1}(\tau)}{E_{2}(\tau)^{3}} = -1 + \cdots. \label{comp2}
\end{align}
For the last three matrices, we note that $E_{g}(\tau) = E_{g, 0}(7\tau)$, where $E_{g, 0}$ is defined in \cite[Page 673]{yang_transf}.
Now, let $\zeta = e^{2\pi i/7}$. The transformation formulas given in \cite[Theorem 1]{yang_transf} imply
\begin{align*}
E_{g}(\tau)\bigg|\pmat{0}{-1}{7}{0} &= E_{g, 0}\left(\frac{-1}{\tau}\right) = e^{\pi i(-1/2 + g/7)}E_{0, -g}(\tau),\\
E_{g}(\tau)\bigg|\pmat{7}{3}{14}{7} &= E_{g, 0}\left(\frac{7\tau + 3}{2\tau + 1}\right) = -ie^{\pi i(2/3 + 3g^{2}/7 - 3g/7)}E_{7g, 3g}(\tau)\\
& = -ie^{\pi i(2/3 + 3g^{2}/7 - 3g/7)}(-1)^{g}\zeta^{-3g^{2}}E_{0, 3g}(\tau),\\
E_{g}(\tau)\bigg|\pmat{7}{2}{21}{7} &= E_{g, 0}\left(\frac{7\tau + 2}{3\tau + 1}\right) = e^{\pi i(-1/6 + 2g^{2}/7 - 2g/7)}E_{7g, 2g}(\tau)\\
&= e^{\pi i(-1/6 + 2g^{2}/7 - 2g/7)}(-1)^{g}\zeta^{-2g^{2}}E_{0, 2g}(\tau).
\end{align*}
Then
\begin{align}
t(\tau)\bigg|\pmat{0}{-1}{7}{0} &= -e^{2\pi i/7}\frac{E_{0, -1}(\tau)^{2}E_{0, -2}(\tau)}{E_{0, -3}(\tau)^{3}} = -\zeta\frac{(1 - \zeta^{-1})^{2}(1 - \zeta^{-2})}{(1 - \zeta^{-3})^{3}} + \cdots\nonumber\\
& = \frac{\sin^{2}(6\pi /7)\sin(2\pi/7)}{\sin^{3}(4\pi/7)} + \cdots =0.15883360\ldots + \cdots,\label{tcusp0}\\
t(\tau)\bigg|\pmat{7}{3}{14}{7} &= -e^{-6\pi i/7}\frac{E_{0, 3}(\tau)^{2}E_{0, 6}(\tau)}{E_{0, 9}(\tau)^{3}}\nonumber\\
&= -\frac{\sin^{2}(4\pi/7)\sin(6\pi/7)}{\sin^{3}(2\pi/7)} + \cdots = -0.86293666\ldots+ \cdots,\label{tcusp1} \\
t(\tau)\bigg|\pmat{7}{2}{21}{7} &= -e^{-4\pi i/7}\frac{E_{0, 2}(\tau)^{2}E_{0, 4}(\tau)}{E_{0, 6}(\tau)^{3}}\nonumber\\
&= -\frac{\sin^{2}(2\pi/7)\sin(4\pi/7)}{\sin^{3}(6\pi/7)} + \cdots = -7.29589694\ldots + \cdots. \label{tcusp2}
\end{align}
In particular, we see that the singularity of the Picard-Fuchs equation at $t \approx 0.15883360$ corresponds to the cusp 0, i.e. $\ld =1/t(0)$, in which $\ld$ is the dominant root of $x^3-5x^3-8x-1$.  Now
$$\eta\left(\frac{-1}{7\tau}\right) = \sqrt{\frac{7\tau}{i}}\eta(7\tau),$$
whereby
\begin{align*}
f\left(\frac{-1}{7\tau}\right) = -\frac{7\zeta^{2}\tau E_{0, -3}(\tau)\eta(7\tau)^{2}}{E_{0, -1}(\tau)^{3}E_{0, -2}(\tau)^{5}} = -Ci\tau(1 + \cdots),
\end{align*}
in which
$$C = \frac{7\sin(4\pi/7)}{128\sin^{3}(6\pi/7)\sin^{5}(2\pi/7)}.$$
In addition, we have
$$f\left(\frac{-1}{7\tau}\right) = \sum_{n = 0}^{\infty}u_{n}t\left(\frac{-1}{7\tau}\right)^{n},$$
which is valid for $\tau$ is a neighborhood containing the upper imaginary axis. For convenience, let $s(\tau) = t(-1/7\tau)$.
Since $\smat{0}{-1}{7}{0}$ normalizes $\Gamma_{1}(7)$, it follows that $s(\tau)$ is a modular function on $\Gamma_{1}(7)$ and thus a rational
function of $t(\tau)$ of degree 1. Likewise, $t$ is a rational function of $s(\tau)$ of degree 1. To be specific,
\begin{align*}
t(\tau) = -\frac{s(2/7) - s(3/7)}{s(\infty) - s(3/7)}\cdot \frac{s(\infty) - s(\tau)}{s(2/7) - s(\tau)},
\end{align*}
because
$t(\infty) = 0$, $t(2/7) = \infty$, and $t(3/7) = - 1$ (from our computation in \eqref{comp1}--\eqref{comp2}).
For $\tau=ix$, $x>0$, our two expressions for $f(-1/7\tau)$ give $\sum_{n = 0}^{\infty}u_{n}s(ix)^{n} = Cx(1 + \cdots).$
Furthermore, we have
$$t(ix) = \frac{E_{1}(ix)^{2}E_{2}(ix)}{E_{3}(ix)^{3}} = e^{-2\pi x}\prod_{n = 1}^{\infty}(1 - e^{-2\pi nx})^{d_{n}},$$
where $d_{n} = 0, 2, 1, -3$ for $n \equiv 0, \pm 1, \pm 2, \pm 3 \imod{7}$, respectively.
Then $x = -(\log t(ix))/(2\pi) + o(1)$ as $x \to \infty$, and it follows that
\begin{align*}
\sum_{n = 0}^{\infty}u_{n}s(ix)^{n} &= -\frac{C}{2\pi}\log t(ix) + o(1) = -\frac{C}{2\pi}\log\left(1 - \frac{s(ix)}{s(\infty)}\right) \\
&\,\,- \frac{C}{2\pi}\log\left(\frac{s(\infty)(s(2/7) - s(3/7))}{(s(\infty) - s(3/7))(s(ix) - s(2/7))}\right) + o(1)
\end{align*}
as $x \rightarrow \infty$. Therefore $f(t) + a\log(1 - t/t(0)) \rightarrow b$ as $t \rightarrow t(0)^{-}$, with
\begin{align*}
a &= \frac{C}{2\pi} =\frac{7\sin(4\pi/7)}{256\pi\sin^{3}(6\pi/7)\sin^{5}(2\pi/7)} \approx 0.3556270700876065 \\
b &= -\frac{C}{2\pi}\log\left(\frac{s(\infty)(s(2/7) - s(3/7))}{(s(\infty) - s(3/7))(s(\infty) - s(2/7))}\right) \approx 0.7144010142820709,
\end{align*}
wherein we have used the identities $t(\tau)|\smat{-1}{0}{0}{1} = t(\tau)$, $s(2/7) = t(-1/2) = t(1/2)$, and $s(3/7) = t(-1/3) = t(1/3)$ as well as the numerical computations in \eqref{tcusp0}--\eqref{tcusp2}.
At this point, an application of Theorem \ref{modform_bound} yields that $u_{n} \sim a\lambda^{n}/n$, in which $\lambda \approx 6.295897$, the dominant root of $x^3-5x^2-8x-1$.

\end{Ex}

\begin{Ex}\label{asym_g8412}
Next, consider the elliptic family associated to $\Gamma(8; 4, 1, 2)$. From \cite{TopYui} we know that the defining equation
is $$y^{2} = x^{3} - 2(8t^{4} - 16t^{3} + 16t^{2} - 8t + 1)x^{2} + (2t - 1)^{4}(8t^{2} - 8t + 1)x.$$
This gives a Picard-Fuchs equation of the form $P_{2}(t)f'' + P_{1}(t)f' + P_{0}f = 0$ where
$P_{2} = t(t-1)(2t - 1)(8t^{2} - 8t + 1)$, $P_{1} = 80t^4 - 160t^3 + 102t^2 - 22t + 1$, and $P_{0} = 4(2t -1)(8t^{2} - 8t + 1)$.
Considering the associated recurrence for $u_{n}$ yields that
$$u_{n} \sim \frac{2}{\pi}\cdot \frac{(4 + 2\sqrt{2})^{n}}{n}.$$
\end{Ex}

\begin{Ex}\label{asym_g110}
Finally, we consider the elliptic family associated to $\Gamma_{1}(10)$ (an index 36 subgroup). Again from \cite{TopYui} we know that the defining equation is
$$y^{2} = x(x^{2} + ax + b)$$ with $a = -(2t^{2} - 2t + 1)(4t^{4} - 12t^{3} + 6t^{2} + 2t - 1)$ and $b = 16t^{5}(t - 1)^{5}(t^{2} - 3t + 1)$.
This gives a Picard-Fuchs equation of the form $P_{2}(t)f'' + P_{1}(t)f' + P_{0}f = 0$ where
$P_{2} = t(t+1)(2t+1)(t^{2} + 3t + 1)(4t^{2} + 2t - 1)$, $P_{1} = 56t^6+240t^5+320t^4+156t^3+12t^2-8t-1$, and $P_{0}  = 2(36t^5+128t^4+136t^3+49t^2+2t-1)$.
Considering the asymptotics of the associated $u_{n}$ gives that
$$u_{n} \sim \frac{2}{\pi\sqrt{5 + 2\sqrt{5}}}\cdot \frac{(1 + \sqrt{5})^{n}}{n}.$$
\end{Ex}

The constants which appear in front of the asymptotics for the recurrences associated to the congruence subgroups of index 24 and 36
are likely related to periods of the elliptic curves associated to the recurrence. 
Moreover, we conjecture that all the constants associated to such asymptotics will be an algebraic number
divided by some half-integer power of $\pi$.

\section{Final Remarks and Future Work}
Systems of the form described in \eqref{boundsrem} have been studied by Beukers \cite{Beukers2} and
Zagier \cite{Zagier} due to their connection with Ap\'ery's proof of the irrationality of $\zeta(3)$ in \cite{Apery};
as such, their attention is restricted to the special case $\deg A=1$ and $\deg B =3$.  We briefly
describe an extension of this system which arises in the case of certain index $24$ congruence subgroups.
Consider \eqref{boundsrem} with instead $\deg A = 3$ and $\deg B = 5$. Thus we have
\begin{align} \label{diffeq}
\left( b_3t^3 + b_2t^2 + b_1t +b_0 \right) F(t) + \left((c_5t^5+c_4t^4+c_3t^3+c_2t^2 -t)F'(t)\right)'=0\;.
\end{align}
If $b_3 = c_5=0$, then the associated holonomic recurrence will have order $3$.  These we choose to ignore,
as none of the recursions in the index $24$ subgroups are of order less than $4$.  Under the parameter change
$t \mapsto -t$, our coefficient series undergoes $u_n \mapsto (-1)^n u_n$, hence we may suppose $c_2 \geq 0$
without loss of generality, and concentrate a brute force search on the domain of about $28.8$ billion $8$-tuples given by
\[ \vert b_3 \vert \leq 10,\; \vert b_2 \vert \leq 25,\; \vert b_1 \vert \leq 15,\; \vert b_0 \vert \leq 5,\; \vert c_5\vert \leq 5,\; \vert c_4 \vert \leq 10,\; \vert c_3 \vert \leq 15,\; \text{ and } 0 \leq c_2 \leq 10.\]
Searching through seems to give evidence to a variant of one of Zagier's conjectures in \cite{Zagier}; that is,
\begin{Conj}
Any integral solution to the differential equation in \eqref{diffeq} corresponds to the solution of a Picard-Fuchs equation about a singular fiber.
\end{Conj}
While in this paper we concentrated on index 24 subgroups, specifically $\Gamma_{1}(7)$,
the authors expect that the same analysis can be applied to $\Gamma_{1}(10)$ and other subgroups of index 36, 48, and 60 though
such analysis might be more technical.

\subsubsection*{Acknowledgements}
The authors would like to thank Dr. Chris Bremer and Dr. Jerome Hoffman of Louisiana State University for giving us guidance and support
during our research experience at LSU in the summer of 2011. The authors would also like to thank Dr. Peter Paule of RISC-Johannes Kepler
University of Linz for giving us access to his algorithmic combinatorics research group's software; Dr. Noam Elkies of Harvard University,
Dr. Yifan Yang of National Chiao Tung University, and Dr. Doron Zeilberger of Rutgers
University for answering multiple questions the authors had, and finally Louisiana State University and the NSF-REU program for providing
the authors with a place to do research. The authors are grateful to the referee for the proof of Example \ref{asym_g17}
and the valuable suggestions for improving this paper.

\bibliographystyle{amsplain}
\bibliography{pf_ref}
\end{document}